\documentclass[11pt,letterpaper]{amsart}

\usepackage{amsmath,amssymb,amsfonts,amsthm,mathrsfs,mathtools}
\usepackage[margin=1.15in]{geometry}
\usepackage{hyperref}
\usepackage{enumitem}
\usepackage{microtype}

\numberwithin{equation}{section}

\newtheorem{theorem}{Theorem}[section]
\newtheorem{lemma}[theorem]{Lemma}
\newtheorem{proposition}[theorem]{Proposition}
\newtheorem{corollary}[theorem]{Corollary}

\theoremstyle{definition}
\newtheorem{definition}[theorem]{Definition}
\newtheorem{remark}[theorem]{Remark}

\newcommand{\R}{\mathbb{R}}
\newcommand{\C}{\mathbb{C}}
\newcommand{\N}{\mathbb{N}}
\newcommand{\Z}{\mathbb{Z}}
\newcommand{\Ree}{\operatorname{Re}}

\newcommand{\e}{\mathrm{e}}

\newcommand{\Res}{\operatorname{Res}}
\newcommand{\Err}{\operatorname{Err}}

\title{Admissibility of H\"ormander--Bernhardsson extremal zeros}
\author{Khai-Hoan Nguyen-Dang}
\address{Morningside Center of Mathematics, Chinese Academy of Sciences, No.\ 55, Zhongguancun East Road, Beijing, 100190, China}
\email{khaihoann@gmail.com}

\date{\today}

\begin{document}
\maketitle

\begin{abstract}
Let $\varphi$ be the H\"ormander--Bernhardsson extremal function, and let $(\pm\tau_n)_{n\ge1}$ be its real zeros. Using the recent analytic description of the zero set ${\tau_n}$, we prove that the squared zeros $\lambda_n=\tau_n^{2}$ form an admissible sequence in the sense of Quine--Heydari--Song: the heat trace $\Theta(t)=\sum_{n\ge1}e^{-\lambda_n t}$ has a full $t\to0^{+}$ expansion in pure powers of $t^{1/2}$. The proof is based on an analytic normal form
\[
\lambda_n=\Bigl(n+\tfrac12\Bigr)^2+q!\Bigl(\Bigl(n+\tfrac12\Bigr)^{-2}\Bigr),
\]
a uniform Taylor expansion in $t$, and a Mellin--Hurwitz zeta analysis of the resulting weighted Gaussian sums. As applications we obtain meromorphic continuation and special-value information for the associated spectral zeta function and zeta-regularized product, sharp large-parameter asymptotics for the canonical product $\prod_{n}(1+z/\lambda_n)$. In particular, we deduce the conjecture by
Bondarenko--Ortega-Cerd\`a--Radchenko--Seip for the special values of the Dirichlet-type series attached to $\varphi$. We also establish a parity dichotomy: sequences $(\tau_n^m)$ are QHS--admissible
for even $m$, while for odd $m$ a nonzero $t\log t$ term obstructs admissibility.
\end{abstract}

\tableofcontents

\section{Introduction}

\subsection{The H\"ormander--Bernhardsson extremal function and its zero set}

A classical theme in complex and harmonic analysis is that extremal problems for
entire functions of prescribed exponential type encode rigid spectral data: zeros,
canonical products, and Dirichlet series that behave as if they came from a
one--dimensional spectrum.  A particularly striking instance is the extremal
problem considered by H\"ormander and Bernhardsson \cite{HB93}: among entire functions
$f$ of exponential type at most $\pi$ satisfying the normalization $f(0)=1$, minimize
the $L^1$--norm on the real line.  Equivalently, in Paley--Wiener language, one looks
for the smallest $PW^1$--norm under the constraint of a fixed point evaluation.
The unique extremizer (up to the natural symmetries) is an even entire function
$\varphi$ of type $\pi$, normalized by $\varphi(0)=1$, with real simple zeros
\[
\{\pm\tau_n\}_{n\ge1}\subset\R,\qquad 0<\tau_1<\tau_2<\cdots\to\infty.
\]
The analysis of this function has recently undergone substantial development in the
works of Bondarenko--Ortega-Cerd\`a--Radchenko--Seip \cite{BOCRS,BOCRS2024_HBI}, who obtain a
factorization $\varphi(z)=\Phi(z)\Phi(-z)$ into canonical products, as well as a
nonlinear functional equation of inversion type intertwining the values of $\Phi$
at $z$ and at $(2\pi i C z)^{-1}$ for an explicit constant $C>0$. The authors also show precise asymptotic information about
the real zeros $\pm\tau_n$ of the H\"ormander--Bernhardsson extremal function.
They indicate that the deviations $\tau_n-(n+\tfrac12)$ are very small in a global sense (in particular, of $\ell^{2}$-type), placing $\{\tau_n\}$ in the classical regime of \emph{stable} perturbations of the half-integer lattice. 

One of the most fruitful byproducts of this structure is the emergence of two Dirichlet
series built from the zeros,
\[
L_+(s):=\sum_{n\ge1}\tau_n^{-s},
\qquad
L_-(s):=\sum_{n\ge1}\frac{(-1)^n}{\tau_n^{\,s}},
\qquad (\Re s>1),
\]
whose analytic continuation, special values, and residues reflect fine information
about $\varphi$ and its functional equation.

The present paper is concerned with a different, and in a sense complementary, aspect
of the same zero set: the \emph{admissibility} of sequences extracted from
$\{\tau_n\}$ in the sense of Quine--Heydari--Song \cite{QHS}.
This notion belongs to the general calculus of zeta regularization and is designed
to capture precisely the analytic input that makes the spectral zeta function and
zeta--regularized products behave as in elliptic spectral theory.
Our goal is to place the H\"ormander--Bernhardsson zeros into that calculus, and then
use the resulting zeta machinery to settle conjecture 1 raised in \cite{BOCRS}.

\subsection{Heat traces, zeta regularization, and admissible sequences}

Given a sequence $(X_n)$ in the right half--plane, one may form the heat trace
\[
H(t):=\sum_{n}e^{-X_n t},\qquad t>0,
\]
together with the associated spectral zeta function
\[
\zeta_X(s):=\sum_n X_n^{-s}.
\]
In geometric and analytic spectral theory, when $(X_n)$ consists of the eigenvalues of a
positive elliptic operator, $H(t)$ admits a full small--time asymptotic expansion and
$\zeta_X(s)$ extends meromorphically to the complex plane. This is one of the standard
routes to the zeta--regularized determinant $\det_\zeta X:=\exp(-\zeta_X'(0))$, which
originates in the work of Ray--Singer and its subsequent refinements
\cite{RaySinger71,Seeley67,Gilkey95,Voros87}.
Quine--Heydari--Song \cite{QHS} abstracted precisely the hypotheses on $(X_n)$ needed to
run this Mellin transform machinery without geometric input.  Their definition
requires the convergence of $H(t)$ for $t>0$, a \emph{full} polyhomogeneous expansion
as $t\to0^+$ in pure powers (no logarithms), and a mild growth condition that precludes
pathological cancellations.  Under these conditions one recovers, in a purely analytic
setting, the familiar package: meromorphic continuation of $\zeta_X(s)$, residue coefficient identities, special value formulae at nonpositive integers, and well--posed zeta regularized products \cite[\S2--\S4]{QHS}.

From this viewpoint it is natural to ask whether the sequences derived from the
H\"ormander--Bernhardsson zeros behave spectrally.  The first obstruction is that the
zero set $\{\tau_n\}$ is \emph{not} a quadratic lattice; rather, it is a delicate
perturbation of the half--integers.  In \cite{BOCRS} it is shown that
\[
\tau_n = n+\frac12 - d_n,
\qquad d_n\in\ell^2,
\]
a perturbative condition of Paley--Wiener type which is powerful for
uniqueness and interpolation in $PW^1$ and for controlling canonical products.
However, $\ell^2$--closeness to a lattice does not by itself imply that the associated
heat trace is polyhomogeneous: logarithmic terms may occur through harmonic weights,
and such logarithms are \emph{forbidden} by the strict Quine--Heydari--Song definition.
Understanding precisely when logarithms appear, and when they do not, is one of the
central points of the present work.

\subsection{Main results: admissibility, parity, and a functional equation for $L_+$}

Our basic spectral object is the squared zero sequence
\[
\lambda_n:=\tau_n^2,\qquad n\ge1,
\]
and the corresponding heat trace
\[
\theta(t):=\sum_{n\ge1}e^{-\lambda_n t}.
\]
Our first main theorem proves that $(\lambda_n)$ is admissible in the sense of
Quine--Heydari--Song.

\begin{theorem}[Heat trace expansion and admissibility]\label{thm:intro-adm}
The heat trace $\theta(t)=\sum_{n\ge1}e^{-\tau_n^2 t}$ admits a full small--time asymptotic
expansion as $t\to0^+$ in \emph{pure} half--integer powers:
\[
\theta(t)\sim \sum_{\nu=0}^\infty c_\nu\,t^{j_\nu},
\qquad j_\nu\in\tfrac12\Z,\quad j_0<j_1<\cdots\to+\infty.
\]
Consequently, $(\tau_n^2)_{n\ge1}$ is admissible in the sense of \cite{QHS}.
\end{theorem}

As in the general QHS calculus \cite{QHS}, this admissibility statement immediately yields
a meromorphic continuation of the spectral zeta function
$Z(s):=\sum_{n\ge1}\lambda_n^{-s}=L_+(2s)$, residue coefficient identities,
and special value formulae at non-positive integers. In addition, it provides a robust
coefficient--extraction mechanism at infinity for the canonical product
$W(\lambda)=\prod_{n\ge1}(1+\lambda/\lambda_n)$.
These consequences form the analytic backbone of our proof of the conjectural symmetry
of $L_+$ discovered in \cite{BOCRS}.

\begin{theorem}[Conjecture~1 of \cite{BOCRS}]\label{thm:intro-conj}
Let $C$ be the constant appearing in the inversion functional equation for $\Phi$ in \cite{BOCRS}.
Then for every $k\in\Z$,
\[
L_+(-2k)=\frac{L_+(2k)}{(2\pi i C)^{2k}}.
\]
\end{theorem}

The proof of Theorem~\ref{thm:intro-conj} illustrates the conceptual role of admissibility
in this problem: the functional equation for the extremal function yields a precise
inversion identity for $\varphi$, which enforces a parity constraint on the even
coefficients in the asymptotic expansion of $\log\varphi(Tz)$ in a small sector
($Tz=(2\pi i C z)^{-1}$). Admissibility supplies the missing analytic input that
identifies those even coefficients with zeta special values at negative integers via
coefficient extraction at infinity.  In this way, the conjectured relation between
$L_+(2k)$ and $L_+(-2k)$ becomes an equality between two independently computable
coefficient functionals.

Beyond squares, we establish a sharp parity phenomenon for power sequences.

\begin{theorem}[Parity dichotomy]\label{thm:intro-parity}
Let $m\ge1$ be an integer and consider $\theta_m(t)=\sum_{n\ge1}e^{-\tau_n^{\,m}t}$.
\begin{enumerate}
\item If $m$ is even, then $(\tau_n^{\,m})_{n\ge1}$ is admissible in the sense of \cite{QHS}.
\item If $m$ is odd, then the small--time expansion of $\theta_m(t)$ contains a \emph{nonzero}
logarithmic term $\kappa_m\,t\log t$. In particular $(\tau_n^{\,m})$ is \emph{not} admissible in the
strict QHS sense.
\end{enumerate}
\end{theorem}

The appearance of logarithms for odd powers is not an artifact of estimation: it is forced by a
specific $a_n^{-1}$ term in the asymptotic expansion of $\tau_n^{\,m}$ (with $a_n=n+\tfrac12$),
and its coefficient can be expressed explicitly through $L_-(m)$, reconciling the heat trace
picture with the residue formula for $L_+$ at negative odd integers recorded in \cite[\S10.2]{BOCRS}.
Thus the failure of admissibility for odd powers has a precise analytic signature:
logarithmic small--time terms correspond to higher--order pole phenomena in Mellin space and to
$\log\lambda$ corrections in large--parameter expansions of resolvent traces and canonical products.
This naturally suggests a \emph{log--polyhomogeneous} enlargement of QHS admissibility, parallel to
the way logarithmic heat terms arise in geometric problems with singularities
\cite{Cheeger83,Gilkey95}.

Finally, our argument yields a general admissibility criterion for analytic perturbations of
quadratic lattices, which should be useful well beyond the H\"ormander--Bernhardsson setting.

\begin{theorem}[Analytic perturbations of quadratic lattices]\label{thm:intro-criterion}
Let $a>0$ and let $q$ be analytic in a neighborhood of $[0,a^{-2}]$.
Then the sequence
\[
\lambda_n=(n+a)^2+q\bigl((n+a)^{-2}\bigr)
\]
is admissible in the sense of \cite{QHS}, provided $\lambda_n>0$ for all $n$.
\end{theorem}

\subsection{Method: analytic normal forms and Mellin analysis of weighted Gaussian sums}

A distinctive feature of the Quine--Heydari--Song notion of admissibility is the \emph{strictly
polyhomogeneous} small-time expansion
\[
\sum_{n}e^{-X_n t}\sim \sum_{\nu\ge 0}c_\nu t^{j_\nu}\qquad(t\to 0^+),
\]
with \emph{no} logarithmic factors $t^\alpha(\log t)^q$ allowed.
This condition is not cosmetic: it is exactly what forces the associated spectral zeta function
$Z(s)=\sum_n X_n^{-s}$ to have at most \emph{simple} poles, and it is what makes the zeta-regularized product calculus behave like a clean Weierstrass theory (e.g.\ with canonical products controlled by finitely many heat coefficients).

The present problem is subtle for a structural reason: the
H\"ormander--Bernhardsson extremal zeros $\tau_n$ are \emph{half-integer perturbations} with a
\emph{borderline} first correction of size $1/n$.  More precisely, writing $a_n:=n+\tfrac12$,
the BOCRS parametrization gives
\[
\tau_n = a_n-\rho(a_n^{-1}),\qquad \rho(z)=\sum_{m\ge1}a_m z^m,\quad a_m\ge0,
\]
so that the leading deviation $a_n-\tau_n$ is of order $a_n^{-1}$ (indeed one has an expansion
$\rho(a_n^{-1})=\frac{a_1}{a_n}+O(a_n^{-3})$).
This is precisely the regime where a term-by-term linearization of the exponential produces
\emph{harmonic weights}:
\[
e^{-\tau_n t}=e^{-a_n t}\exp\!\bigl((a_n-\tau_n)t\bigr)
= e^{-a_n t}\Bigl(1+t\,\frac{a_1}{a_n}+O(t a_n^{-3})+O(t^2 a_n^{-2})\Bigr).
\]
Summing the first correction over $n$ yields the classical logarithmic amplification
\[
t\sum_{n\ge 1}\frac{e^{-a_n t}}{a_n}
\sim t\log\!\frac1t \qquad (t\to 0^+),
\]
so even though the pointwise perturbation $a_n-\tau_n$ is small, the \emph{collective} effect of
the borderline weight $a_n^{-1}$ creates a genuine $t\log t$ term in the heat trace.
In other words: the obstruction is not local in $n$ but global in the Laplace summation, and it is
invisible unless one tracks the exact first nonsummable correction.  This is exactly why
BOCRS--admissibility (an $\ell^2$ perturbation of $a_n$) does \emph{not} automatically imply
QHS--admissibility (a pure-power heat expansion).

The resolution is to identify a normal form in which the offending harmonic weight cannot occur.
For the squared zeros
\[
\lambda_n:=\tau_n^2,
\]
the oddness of $\rho$ implies the analytic reparametrization
\[
\tau_n=a_n-a_n^{-1}\sigma(a_n^{-2})
\quad\Longrightarrow\quad
\lambda_n=a_n^2+q(a_n^{-2}),
\]
with $q$ analytic on a neighborhood of $[0,r_0]$ (and in particular bounded there).
This is the key cancellation: \emph{squaring removes the $a_n^{-1}$ term} and replaces it by an
analytic correction in $u_n:=a_n^{-2}$.
Consequently, in the heat trace
\[
\theta(t)=\sum_{n\ge1}e^{-\lambda_n t}
=\sum_{n\ge1}e^{-a_n^2 t}\,e^{-t q(u_n)},
\]
we never encounter the harmonic weight $a_n^{-1}$ that produces $t\log t$.

The proof then has two robust components.

\smallskip
\noindent\textbf{(i) Taylor expansion in time with a summable remainder.}
We expand $e^{-tq(u)}$ in powers of $t$ \emph{uniformly} for $u\in[0,4/9]$,
without expanding $q$ in $u$.  This yields, for any truncation order $R$,
\[
\theta(t)=\sum_{r=0}^R \frac{(-t)^r}{r!}\,F_r(t) \;+\; O(t^{R+1})\Theta_0(t),
\qquad
F_r(t):=\sum_{n\ge1}q(u_n)^r e^{-a_n^2 t},
\]
so the error inherits the sharp $t^{-1/2}$ growth of the model Gaussian sum
$\Theta_0(t)=\sum_{n\ge1}e^{-a_n^2 t}$.

\smallskip
\noindent\textbf{(ii) Mellin transform and contour shifting for weighted Gaussian sums.}
For fixed $r$, the weight $u\mapsto q(u)^r$ is analytic and admits an absolutely convergent power series on $[0,4/9]$. This reduces $F_r(t)$ to a convergent linear combination of the weighted theta sums
\[
\Theta_m(t):=\sum_{n\ge1}a_n^{-2m}e^{-a_n^2 t}.
\]
We then study $\Theta_m(t)$ via Mellin inversion and contour shifting:
\[
\Theta_m(t)=\frac{1}{2\pi i}\int_{\Re s=c}\Gamma(s)\,t^{-s}
\Bigl(\zeta(2(m+s),\tfrac12)-2^{2(m+s)}\Bigr)\,ds,
\]
using the meromorphic structure of $\Gamma$ and the Hurwitz zeta function.
Since all poles are simple and do not collide in the even--power regimes, one obtains a
full asymptotic expansion in half--integer powers of $t$ with controlled remainder,
and the analyticity of $q$ guarantees that the ensuing infinite $m$--summations are
geometrically convergent and may be performed termwise.  This produces the full expansion
for $\theta(t)$ and hence QHS admissibility.

For odd powers $\tau_n^{2\ell+1}$ the same Mellin mechanism shows precisely how a logarithm is created:
a harmonic weight $a_n^{-1}$ appears in the large--$n$ expansion of $\tau_n^{2\ell+1}$,
and the corresponding weighted Gaussian sum has leading asymptotic $\frac1m\log(1/t)$, which forces
a nonzero $t\log t$ term in the heat trace.

The decisive point is that for $\lambda_n=\tau_n^2$ (and more generally for
even powers $\tau_n^{2p}$) the spectrum admits the analytic normal form
\[
\lambda_n=(n+a)^2+q\bigl((n+a)^{-2}\bigr),
\]
which rules out harmonic weights and allows a purely polyhomogeneous Mellin calculus.

The analytic engine behind admissibility proofs is the heat trace
$\Theta(t)$ and its small-$t$ asymptotics.
Standard references for heat kernel expansions and their spectral meaning are
\cite{Gilkey95,BerlineGetzlerVergne1992_HeatKernelsDirac,Vassilevich2003_HeatKernelUsersManual}.
When logarithmic terms can arise (typically in singular or boundary-value settings),
Cheeger's analysis provides a useful benchmark \cite{Cheeger83}.

To pass from heat asymptotics to zeta information, we rely on Mellin transform and
contour-shift techniques. For sums of theta type and harmonic-sum asymptotics, a
particularly effective toolkit is the Mellin-asymptotics framework of
Flajolet--Gourdon--Dumas \cite{FlajoletGourdonDumas1995_MellinHarmonicSums}, together
with Mellin--Barnes technology and classical asymptotic-analysis references
\cite{ParisKaminski2001_MellinBarnes,Olver_AsymptoticsSpecialFunctions,Wong2001_AsymptoticIntegrals}
and standard special-function identities (e.g.\ Hurwitz zeta) as organized in the
DLMF \cite{NISTDLMF}.
\subsection{Relation to other works and scope}

The extremal function $\varphi$ belongs to a long line of Fourier--analytic extremal problems for bandlimited functions, with origins in the Paley--Wiener theory \cite{PaleyWiener34} and the de Branges theory of Hilbert spaces of entire functions \cite{deBranges68}.
The H\"ormander--Bernhardsson minimization problem \cite{HB93} is a distinguished $PW^1$ analogue whose solution displays unusually rigid structure. T he recent analysis in \cite{BOCRS} uncovers
new algebraic symmetries and analytic continuations of the associated Dirichlet series. Earlier stages of this program appear in \cite{BOCRS2024_HBI}, which establish
quantitative control on the zero set $\{\pm\tau_n\}$ of $\varphi$, including an $\ell^{2}$-type proximity of $\tau_n$ to the half-integers.
Complementary viewpoints on the same extremizer theme in Paley--Wiener
and de Branges settings have emerged very recently in \cite{BrevigChirreOrtegaCerdaSeip2024_PointEvalPW,Instanes2024_PWPointEval,Bergman2024_deBrangesPointEval}.
In particular, these works develop an operator- and space-theoretic framework for
point evaluation in $PW^{p}$ and in model spaces, providing tools that are
useful when one needs stability statements for extremizers and their associated
spectral data.

The notion of admissibility in this paper is formulated in the language of zeta-regularized products. A systematic framework for regularized products of sequences is given by Quine--Heydari--Song \cite{QHS},
which isolates analytic conditions on sequences that guarantee that the associated zeta functions admit meromorphic continuation and that the corresponding
regularized products can be defined by $\exp(-\zeta'(0))$.
This sequence calculus aligns with the spectral dictionary developed in the theory
of elliptic operators: Seeley's complex powers \cite{Seeley67} and the Ray--Singer
analytic torsion program \cite{RaySinger71} explain how heat-trace asymptotics control
the meromorphic continuation of spectral zeta functions and determinants.

For context and comparison with broader zeta regularization theories (and
with anomaly phenomena), we also refer to the classical works of Voros and
Kontsevich--Vishik \cite{Voros87,KontsevichVishik1994_DeterminantsPsiDO}, as well as
to the monograph treatments \cite{JorgensonLang1993_RegularizedSeriesProducts,Paycha2012_RegularisedIntegralsSumsTraces,Kirsten2001_SpectralFunctions,ElizaldeEtAl1994_ZetaRegularization}.

\subsection{On notions of admissibility}
The adjective \emph{admissible} refers to genuinely different hypotheses in the two settings at hand.

In the BOCRS framework one considers a strictly increasing sequence
$t=(t_n)_{n\ge 1}\subset(0,\infty)$ and calls it \emph{admissible} if
\[
d_n := n+\tfrac12-t_n \in \ell^2.
\]
This is a perturbative Paley--Wiener type condition: it guarantees that the
canonical products built from $t$, e.g.
\[
\psi_t(z):=\prod_{n\ge 1}\Bigl(1-\frac{z^2}{t_n^2}\Bigr),
\qquad
\Psi_t(z):=\prod_{n\ge 1}\Bigl(1+(-1)^n\frac{z}{t_n}\Bigr),
\]
have essentially the same exponential type and growth behavior as the corresponding
half--integer model products, and it supports interpolation arguments
for $PW^1$.

In the QHS framework one starts from a sequence $(X_k)$ with $\Re X_k>0$ and assumes
that the heat trace
\[
H(t):=\sum_k e^{-X_k t}
\]
converges for $t>0$ and admits a \emph{full} asymptotic expansion as $t\to 0^+$ of the form
\[
H(t)\sim \sum_{\nu=0}^\infty c_{j_\nu}\,t^{j_\nu},
\qquad j_0<j_1<\cdots\to +\infty,
\]
together with a mild small--$t$ growth condition (e.g.\ $t^\beta H(t)\to 0$ for some
$\beta>0$).
Crucially, this hypothesis is \emph{purely polyhomogeneous} (no logarithmic terms),
and it is tailored to Mellin-transform arguments yielding meromorphic continuation of
the associated zeta function and zeta-regularized products.

BOCRS--admissibility is a \emph{metric} condition controlling how close the points $t_n$
are to $n+\tfrac12$ (an $\ell^2$ perturbation), whereas QHS--admissibility is an
\emph{asymptotic} condition controlling the small-time structure of $H(t)$.
Consequently, neither notion implies the other in general:
\begin{itemize}
\item A QHS--admissible sequence need not be BOCRS--admissible (e.g.\ $X_k=k^2$ has a
standard heat-trace expansion, but $k^2$ is not an $\ell^2$ perturbation of $k+\tfrac12$).
\item A BOCRS--admissible sequence need not be QHS--admissible.  In particular, for the
H\"ormander--Bernhardsson zero sequence $\tau_n$ one has an expansion of the form
\[
\tau_n = n+\tfrac12-\frac{a_1}{n+\tfrac12}+O\bigl(n^{-3}\bigr),
\]
so that $n+\tfrac12-\tau_n\asymp 1/n$.  Expanding
$e^{-\tau_n t}=e^{-(n+\tfrac12)t}\exp\bigl((n+\tfrac12-\tau_n)t\bigr)$ yields a correction
term proportional to
\[
t\sum_{n\ge 1}\frac{e^{-(n+\tfrac12)t}}{n+\tfrac12}\sim t\log(1/t),
\]
which produces a $t\log t$ contribution in $H(t)$ and therefore violates the strict
QHS requirement of a pure-power asymptotic expansion.
\end{itemize}

If one wishes to encompass sequences whose heat traces involve terms
$t^\alpha(\log t)^q$ (as happens naturally for odd-power reparametrizations such as
$X_k=\tau_k^{2m+1}$), a natural enlargement of QHS--admissibility is to allow
\emph{log-polyhomogeneous} expansions.  On the zeta-function side this corresponds to
allowing higher-order poles (simple poles correspond to pure powers), and one then
defines regularized products using zeta-renormalization in that broader
setting.

\subsection{Future directions}
\label{subsec:portability-contexts}

The analytic normal form
\[
\lambda_n=(n+a)^2+q\bigl((n+a)^{-2}\bigr)\qquad(q\ \text{analytic near }[0,a^{-2}])
\]
is not tied to the H\"ormander--Bernhardsson extremal problem \cite{BOCRS}, it is a common large-parameter structure in several areas.  Here are three concrete directions where the same mechanism is expected to apply and where the present admissibility criterion can be used essentially verbatim.

\begin{enumerate}[label=\textup{(\arabic*)},leftmargin=2.6em]
\item \textbf{One-dimensional spectral problems (Sturm--Liouville, Schr\"odinger).}
For regular boundary value problems on an interval, eigenvalues satisfy
\[
\lambda_n \sim (n+a)^2 + \sum_{j\ge1}c_j(n+a)^{-2j},
\]
with coefficients determined by the boundary conditions and the potential (Liouville--Green asymptotics); see, for instance,
\cite{Zettl2005_SturmLiouvilleTheory,LevitanSargsjan1991_SLDirac,PoschelTrubowitz1987_InverseSpectralTheory}.
When the asymptotic series is controlled analytically (e.g.\ for analytic potentials or
in integrable families), the correction can often be organized as an analytic function of $(n+a)^{-2}$,
placing the spectrum in the scope of the analytic perturbation of a quadratic lattice framework.
A direct payoff is a clean heat expansion and thus determinant-type information; compare \cite{GelfandYaglom1960_FunctionalSpacesDet,BFK1995_LineSegmentDet,KirstenMcKane2004_DeterminantsSL}.

\item \textbf{Special-function spectra (Bessel and Airy-type zeros and their zeta functions).}
Many canonical products in analysis and mathematical physics are built from zeros $x_{n}$ of special
functions, with large-$n$ expansions
\[
x_n \sim (n+a)\pi + \sum_{j\ge1}\frac{\alpha_j}{(n+a)^{2j-1}}.
\]
After squaring, one typically obtains
\[
x_n^2 \sim \pi^2 (n+a)^2 + \widetilde q\bigl((n+a)^{-2}\bigr),
\]
where $\widetilde q$ is analytic near $0$ under suitable uniformity assumptions.
This is precisely the analytic-quadratic-lattice paradigm and leads to QHS-style admissibility
(compare \cite{QHS,Seeley67}), meromorphic continuation of the associated zeta functions, and controlled canonical products; see representative Bessel-zeta treatments in
\cite{ActorBender1996_BesselZerosZeta,Spreafico2004_BesselZeta,NaberBruckCostello2022_BesselZeta,AldanaKirstenRowlett2025_VariationBarnesBessel}.

\item \textbf{de Branges sampling problems with quadraticized lattices.}
Zeros of Hermite--Biehler functions (or sine-type functions) frequently form near-lattices
$n+a+\varepsilon_n$ in Paley--Wiener settings; see \cite{deBranges68,Seip2004_InterpolationSampling}.
In problems where one naturally passes to squares (or even powers) of such zeros
(e.g.\ via even canonical products, spectral transforms, or functional equations), the resulting
sequence often becomes
\[
(n+a)^2 + q\bigl((n+a)^{-2}\bigr)
\]
with $q$ analytic when the perturbation is produced by an analytic reparametrization.
This provides a portable route to heat-trace asymptotics and zeta-regularized product formulae in
Paley--Wiener settings, parallel to (but distinct from) purely $\ell^2$ perturbation theory:
compare the density viewpoint
\cite{Landau1967_Density,OrtegaCerdaSeip2002_FourierFrames,MarzoNitzanOlsen2012_DoublingPhase}
with sharp perturbation criteria for exponential Riesz bases
\cite{Kadec1964_PaleyWienerConstant,Avdonin1974_RieszBases,Pavlov1979_BasicityMuckenhoupt}.
\end{enumerate}

\noindent
In all three contexts, the same principle governs whether logarithms appear: odd-power (or unsquared) perturbations may leave a residual $(n+a)^{-1}$ harmonic weight and hence
create $t\log t$, while even-power quadraticization tends to eliminate that term and restore a pure-power
heat calculus; see, e.g., the regular-singular heat expansions in\cite{BruningSeeley1985_RegularSingularAsymptotics,BruningSeeley1987_ResolventExpansion}.

\subsection*{Organization of the paper}

Section~\ref{sec:QHS} recalls the notion of admissibility in the sense of Quine--Heydari--Song and establishes the basic setup for the squared zero sequence $\lambda_n=\tau_n^2$.  In particular, we derive the analytic normal form $\lambda_n=a_n^2+q(a_n^{-2})$ and reduce the heat trace $\theta(t)=\sum_{n\ge1}e^{-\lambda_n t}$ to weighted Gaussian sums (Section~\ref{sec:reduction}).  The Mellin--Hurwitz zeta analysis of these weighted sums produces full small-time asymptotics (Sections~\ref{subsec:Theta-asymp} and \ref{subsec:Fr-asymp}), and the admissibility of $(\lambda_n)$ is concluded in Section~\ref{sec:final}.

Section~\ref{sec:consequences} records the resulting spectral-zeta consequences of admissibility, including meromorphic continuation and special-value identities for the spectral zeta function, as well as coefficient extraction statements and large-parameter asymptotics for the associated canonical product.  In Section~\ref{sec:conj1} we combine these admissibility consequences with the inversion identity for $\varphi$ to prove Conjecture~1 of \cite{BOCRS}.  Section~\ref{sec:parity-dichotomy} establishes the parity dichotomy for $\theta_m(t)=\sum_{n\ge1}e^{-\tau_n^{,m}t}$, showing that even powers yield admissible sequences while odd powers force a nonzero $t\log t$ obstruction.  Finally, Section~\ref{sec:parity-consequences} discusses the corresponding zeta-theoretic and canonical-product consequences in the log-polyhomogeneous (odd-power) regime.

\subsection*{Acknowledgements}

We thank Quoc-Hung Nguyen for bringing this problem and conjecture to our attention. We are grateful to Vu Quang Huynh for valuable guidance on zeta regularization and for related discussions during our undergraduate studies. We thank Kristian Seip for his encouragement and correspondence. We are grateful to the Morningside Center of Mathematics, Chinese Academy of Sciences, for its support and a stimulating research environment.

\section{On admissible sequences in the sense of Quine--Heydari--Song}\label{sec:QHS}

\subsection{Setting}

We recall the definition from \cite[\S2]{QHS}.

\begin{definition}[Admissible sequence]\label{def:admissible}
A sequence $(X_k)_{k\ge 0}\subset \C\setminus\{0\}$ is called \emph{admissible} if:
\begin{enumerate}[label=\textup{(\arabic*)},leftmargin=2.5em]
\item $X_k$ lies in the half-plane $\Ree X>0$ for every $k$;
\item the series $\sum_{k\ge 0}\e^{-X_k t}$ converges absolutely for every $t>0$ and has
a full asymptotic expansion as $t\to 0^+$ of the form
\[
\sum_{k\ge 0}\e^{-X_k t}\sim \sum_{\nu=0}^\infty c_\nu t^{j_\nu},
\qquad j_0<j_1<\cdots,\quad j_\nu\to+\infty;
\]
\item there exists $\beta>0$ such that
\[
\lim_{t\to 0^+} t^\beta \sum_{k\ge 0}\e^{-|X_k|t} = 0.
\]
\end{enumerate}
\end{definition}

\begin{remark}
Quine--Heydari--Song index their sequences from $k=0$, whereas we will naturally work with
$n\ge 1$. Adding or removing finitely many terms does not affect any of the three conditions
in Definition \ref{def:admissible}, so we freely switch between these conventions.
\end{remark}

Throughout this paper, \emph{admissible} always means admissible in the sense of Quine--Heydari--Song.

\subsection{On conditions (1) and (3)}

Let $a_n:=n+\tfrac12$. The following is a reformulation of results proved in
\cite[Theorem~6.1 and Theorem~1.2]{BOCRS}.

\begin{theorem}[Convergent power series representation of $\tau_n$]\label{thm:rho}
There exists an odd power series
\[
\rho(z)=\sum_{m\ge 1} a_m z^m,\qquad a_m\ge 0,
\]
such that:
\begin{enumerate}[label=\textup{(\arabic*)},leftmargin=2.5em]
\item the series $\rho$ converges if and only if $|z|\le 2$;
\item $\rho(2)=\tfrac12$;
\item for every $n\ge 1$,
\begin{equation}\label{eq:tau-rho}
\tau_n = a_n - \rho(a_n^{-1}) = n+\frac12 - \rho\!\Bigl(\frac{1}{n+\frac12}\Bigr).
\end{equation}
\end{enumerate}
\end{theorem}

\begin{remark}
Since the coefficients of $\rho$ are nonnegative and the series converges at $z=2$,
it converges absolutely on $[0,2]$ and defines a continuous increasing function there.
\end{remark}

Throughout we set $\lambda_n := \tau_n^2$ and
\[
\theta(t) := \sum_{n\ge 1}\e^{-\lambda_n t}.
\]

\begin{lemma}[Lower bound for $\tau_n$]\label{lem:tau-lower}
For every $n\ge 1$ one has $\tau_n\ge n$. Consequently $\lambda_n\ge n^2$.
\end{lemma}

\begin{proof}
Since $\rho$ has nonnegative coefficients and converges at $z=2$ (so $\sum_{j\ge1} \rho_j 2^j<\infty$),
the series defining $\rho$ converges absolutely and uniformly on $[0,2]$ by the Weierstrass $M$-test.
Hence $\rho$ is continuous on $[0,2]$, and because each partial sum is increasing on $[0,2]$,
the limit function $\rho$ is increasing on $[0,2]$.

By Theorem \ref{thm:rho}, $\rho$ has nonnegative coefficients and hence is increasing on $[0,2]$.
Since $a_n\ge \tfrac32$, we have $a_n^{-1}\le \tfrac23<2$, so
\[
0\le \rho(a_n^{-1})\le \rho(2)=\frac12.
\]
Using \eqref{eq:tau-rho} yields
\[
\tau_n = a_n - \rho(a_n^{-1}) \ge \Bigl(n+\frac12\Bigr) - \frac12 = n,
\]
and squaring gives $\lambda_n\ge n^2$.
\end{proof}

\begin{corollary}\label{cor:theta-basic}
The series $\theta(t)$ converges absolutely for every $t>0$ and satisfies
$\theta(t)=O(t^{-1/2})$ as $t\to 0^+$.
In particular, Definition \ref{def:admissible}\textup{(3)} holds for $\beta=1$.
\end{corollary}

\begin{proof}
By Lemma \ref{lem:tau-lower}, $\theta(t)\le \sum_{n\ge 1}\e^{-n^2 t}$.
For $t>0$,
\[
\sum_{n\ge 1}\e^{-n^2 t}\le \int_0^\infty \e^{-x^2 t}\,dx=\frac12\sqrt{\frac{\pi}{t}},
\]
which implies $\theta(t)=O(t^{-1/2})$.
Thus $t\,\theta(t)=O(t^{1/2})\to 0$ as $t\to 0^+$, establishing Definition \ref{def:admissible}(3).
\end{proof}

\subsection{Reduction to weighted Gaussian sums}\label{sec:reduction}

It remains to prove Definition \ref{def:admissible}(2), i.e.\ the full small-time asymptotic expansion. We exploit the odd nature of $\rho$ to rewrite $\lambda_n$ as a quadratic main term plus an analytic correction.

Write $\rho$ as an odd series
\[
\rho(z)=\sum_{j\ge 0} b_j z^{2j+1},\qquad b_j:=a_{2j+1}\ge 0,
\]
and define
\[
\sigma(u):=\sum_{j\ge 0} b_j u^j.
\]
Then $\rho(z)=z\,\sigma(z^2)$, so for $x>0$,
\begin{equation}\label{eq:rho-sigma}
\rho(1/x)=x^{-1}\sigma(x^{-2}).
\end{equation}

Since $\rho$ converges for $|z|\le 2$, the series $\sigma$ converges for $|u|\le 4$.
Define the analytic function
\begin{equation}\label{eq:def-q}
q(u):=-2\sigma(u)+u\sigma(u)^2,\qquad |u|\le 4.
\end{equation}

\begin{lemma}[Analytic correction for $\lambda_n$]\label{lem:lambda-q}
Let $a_n=n+\tfrac12$ and $u_n=a_n^{-2}$. Then
\begin{equation}\label{eq:lambda-q}
\lambda_n=\tau_n^2=a_n^2 + q(u_n).
\end{equation}
Moreover, $(u_n)_{n\ge 1}\subset (0,4/9]$.
\end{lemma}

\begin{proof}
By \eqref{eq:tau-rho} and \eqref{eq:rho-sigma},
\[
\tau_n=a_n-a_n^{-1}\sigma(a_n^{-2}) = a_n-a_n^{-1}\sigma(u_n).
\]
Squaring gives
\[
\tau_n^2 = a_n^2 -2\sigma(u_n)+u_n\sigma(u_n)^2 = a_n^2 + q(u_n),
\]
which is \eqref{eq:lambda-q}.
Finally, $u_n=a_n^{-2}\le (3/2)^{-2}=4/9$ for $n\ge 1$.
\end{proof}

Recall that by Lemma~\ref{lem:lambda-q} we have
\[
\lambda_n=a_n^2+q(u_n),\qquad a_n=n+\frac12,\qquad u_n=a_n^{-2}\in(0,r_0],
\quad r_0:=\frac49,
\]
and hence
\begin{equation}\label{eq:theta-split-correct}
\theta(t)=\sum_{n\ge 1}\e^{-a_n^2 t}\,\e^{-tq(u_n)}.
\end{equation}
In this part we reduce the asymptotic analysis of $\theta(t)$ as $t\to 0^+$ to the
basic weighted Gaussian sums
\[
\Theta_m(t):=\sum_{n\ge 1}a_n^{-2m}\e^{-a_n^2 t},\qquad m\ge 0.
\]
The crucial point is that we \emph{do not} truncate $q$ in the variable $u$, we only
perform a uniform Taylor expansion in the time variable $t$, which produces an
error of size $t^{R+1}\Theta_0(t)$ and hence ultimately $O(t^{R+1/2})$.

\subsubsection{Uniform Taylor expansion in $t$}\label{subsec:t-taylor}

Since $q$ is analytic in a neighborhood of $[0,r_0]$, it is bounded there. Set
\[
Q_*:=\sup_{u\in[0,r_0]}|q(u)|<\infty.
\]

\begin{lemma}[Uniform Taylor remainder for the exponential]\label{lem:exp-taylor}
Let $R\in\Z_{\ge0}$. Then for all $x\in\C$ and all $t\in[0,1]$,
\begin{equation}\label{eq:exp-taylor}
    e^{-tx}=\sum_{r=0}^{R}\frac{(-tx)^r}{r!}+O_R\!\bigl(t^{R+1}|x|^{R+1}e^{t|x|}\bigr).
\end{equation}
\end{lemma}

\begin{proof}
By the exponential series,
\[
e^{-tx}-\sum_{r=0}^{R}\frac{(-tx)^r}{r!}
=\sum_{n=R+1}^{\infty}\frac{(-tx)^n}{n!}.
\]
Taking absolute values and using $t\in[0,1]$ gives
\[
\left|e^{-tx}-\sum_{r=0}^{R}\frac{(-tx)^r}{r!}\right|
\le \sum_{n=R+1}^{\infty}\frac{(t|x|)^n}{n!}
=\frac{(t|x|)^{R+1}}{(R+1)!}\sum_{k=0}^{\infty}\frac{(t|x|)^k}{(R+2)\cdots(R+1+k)}.
\]
Since $(R+2)\cdots(R+1+k)\ge k!$, we obtain
\[
\sum_{n=R+1}^{\infty}\frac{(t|x|)^n}{n!}
\le \frac{(t|x|)^{R+1}}{(R+1)!}\sum_{k=0}^{\infty}\frac{(t|x|)^k}{k!}
=\frac{(t|x|)^{R+1}}{(R+1)!}e^{t|x|},
\]
which yields the claimed bound.
\end{proof}

Define for $r\ge 0$ and $t>0$ the auxiliary sums
\begin{equation}\label{eq:def-Fr}
F_r(t):=\sum_{n\ge 1}q(u_n)^r\,\e^{-a_n^2 t}.
\end{equation}
These series converge absolutely for every $t>0$ since $|q(u_n)|\le Q_*$ and
$\sum_{n\ge 1}\e^{-a_n^2 t}<\infty$.

\begin{proposition}[Finite reduction in $t$]\label{prop:reduction}
Fix an integer $R\ge 0$. Then for $t\in(0,1]$,
\begin{equation}\label{eq:theta-reduction-correct}
\theta(t)=\sum_{r=0}^{R}\frac{(-t)^r}{r!}\,F_r(t)\;+\;\mathrm{Err}_R(t),
\end{equation}
where the error satisfies
\begin{equation}\label{eq:ErrR}
\mathrm{Err}_R(t)=O(t^{R+1})\,\Theta_0(t).
\end{equation}
\end{proposition}

\begin{proof}
Insert \eqref{eq:exp-taylor} with $u=u_n$ into \eqref{eq:theta-split-correct}:
\[
\e^{-tq(u_n)}=\sum_{r=0}^{R}\frac{(-t)^r}{r!}\,q(u_n)^r+O(t^{R+1}),
\]
uniformly in $n$ for $t\in(0,1]$. Multiplying by $\e^{-a_n^2 t}$ and summing over $n\ge 1$
gives \eqref{eq:theta-reduction-correct} with
\[
\mathrm{Err}_R(t)
=\sum_{n\ge 1}\e^{-a_n^2 t}\,O(t^{R+1})
=O(t^{R+1})\sum_{n\ge 1}\e^{-a_n^2 t}
=O(t^{R+1})\Theta_0(t),
\]
which is \eqref{eq:ErrR}.
\end{proof}

\subsubsection{Analytic expansion in $u$ and reduction to $\Theta_m$}\label{subsec:u-expansion}

For each $r\ge 0$, the function $u\mapsto q(u)^r$ is analytic in $|u|<4$ and hence admits
a power series expansion around $u=0$:
\begin{equation}\label{eq:q^r-series}
q(u)^r=\sum_{m=0}^\infty c_{r,m}\,u^m,\qquad |u|<4,
\end{equation}
with real coefficients $c_{r,m}$ (since $q$ has real Taylor coefficients).

\begin{lemma}[Geometric bound for the coefficients $c_{r,m}$]\label{lem:crm-bound}
Fix $r\ge 0$ and choose any $\rho$ with $r_0<\rho<4$. Then there exists $C_{r,\rho}>0$
such that
\begin{equation}\label{eq:crm-bound}
|c_{r,m}|\le C_{r,\rho}\,\rho^{-m}\qquad(m\ge 0).
\end{equation}
In particular, the series $\sum_{m\ge 0}|c_{r,m}|\,r_0^m$ converges.
\end{lemma}

\begin{proof}
By Cauchy's estimate applied to the analytic function $q(u)^r$ on the circle $|u|=\rho$,
\[
|c_{r,m}|=\left|\frac{1}{2\pi i}\int_{|u|=\rho}\frac{q(u)^r}{u^{m+1}}\,du\right|
\le \rho^{-m}\sup_{|u|=\rho}|q(u)|^r.
\]
This gives \eqref{eq:crm-bound} with $C_{r,\rho}:=\sup_{|u|=\rho}|q(u)|^r$.
Since $r_0/\rho<1$, the final claim follows.
\end{proof}

\begin{proposition}[Expansion of $F_r(t)$ in terms of $\Theta_m(t)$]\label{prop:Fr-Theta}
For every $r\ge 0$ and $t>0$,
\begin{equation}\label{eq:Fr-Theta}
F_r(t)=\sum_{m=0}^\infty c_{r,m}\,\Theta_m(t),
\end{equation}
and the series on the right converges absolutely.
\end{proposition}

\begin{proof}
Since $u_n\in(0,r_0]$ and \eqref{eq:q^r-series} converges absolutely on $[0,r_0]$,
we have $q(u_n)^r=\sum_{m\ge 0}c_{r,m}u_n^m=\sum_{m\ge 0}c_{r,m}a_n^{-2m}$.
Hence,
\[
F_r(t)=\sum_{n\ge 1}\left(\sum_{m\ge 0}c_{r,m}a_n^{-2m}\right)\e^{-a_n^2 t}.
\]
To justify interchanging sums, note that for each fixed $t>0$,
\[
\sum_{n\ge 1}\sum_{m\ge 0}|c_{r,m}|\,a_n^{-2m}\e^{-a_n^2 t}
\le \sum_{m\ge 0}|c_{r,m}|\left(\sup_{n\ge 1}a_n^{-2m}\right)\sum_{n\ge 1}\e^{-a_n^2 t}
= \Theta_0(t)\sum_{m\ge 0}|c_{r,m}|\,r_0^m,
\]
and the right-hand side is finite by Lemma~\ref{lem:crm-bound}. Tonelli's theorem
therefore applies and yields \eqref{eq:Fr-Theta}.
\end{proof}

\medskip
Combining Propositions~\ref{prop:reduction} and \ref{prop:Fr-Theta}, we obtain for any
$R\ge 0$ and $t\in(0,1]$ the representation
\begin{equation}\label{eq:theta-double}
\theta(t)
=\sum_{r=0}^{R}\frac{(-t)^r}{r!}\sum_{m=0}^\infty c_{r,m}\,\Theta_m(t)
\;+\;O(t^{R+1})\Theta_0(t).
\end{equation}

\subsection{Full asymptotics for $\Theta_m(t)$}\label{subsec:Theta-asymp}
We will use the Hurwitz zeta function $\zeta(s,a)=\sum_{n\ge 0}(n+a)^{-s}$ for $\Ree s>1$.

\begin{lemma}[Half-integer identity]\label{lem:hurwitz-half}
For $\Ree p>1$,
\[
\sum_{n\ge 0}(n+\tfrac12)^{-p}=\zeta(p,\tfrac12)=(2^p-1)\zeta(p).
\]
Hence, as meromorphic functions of $p$,
\begin{equation}\label{eq:hurwitz-identity}
\zeta(p,\tfrac12)=(2^p-1)\zeta(p).
\end{equation}
\end{lemma}

\begin{proof}
For $\Ree p>1$,
\[
\sum_{n\ge 0}(n+\tfrac12)^{-p}
=2^p\sum_{n\ge 0}(2n+1)^{-p}
=2^p\Bigl(\sum_{m\ge 1}m^{-p}-\sum_{n\ge 1}(2n)^{-p}\Bigr)
=(2^p-1)\zeta(p).
\]
Both sides extend meromorphically to $\C$, hence \eqref{eq:hurwitz-identity} follows by analytic continuation.
\end{proof}

Recall
\[
\Theta_m(t)=\sum_{n\ge 1} a_n^{-2m}\e^{-a_n^2 t}
=\sum_{n\ge 1} (n+\tfrac12)^{-2m}\e^{-(n+\tfrac12)^2 t}.
\]

\begin{proposition}[Full asymptotic expansion of $\Theta_m(t)$]\label{prop:Theta-asymp}
For each integer $m\ge 0$, the function $\Theta_m(t)$ admits a full asymptotic expansion
as $t\to 0^+$ with exponents in $\tfrac12\Z$.

More precisely, fix an integer $K\ge 0$ and set $s_K:=-K-\tfrac12$. Then there exist
constants $A_{m,j}$ ($0\le j\le K$) such that, as $t\to 0^+$,
\begin{equation}\label{eq:Theta-K}
\Theta_m(t)
=\frac{\sqrt{\pi}}{2}\,\mathbf{1}_{\{m=0\}}\,t^{-1/2}
\;+\;\sum_{j=0}^{K}A_{m,j}\,t^{j}
\;+\;\mathbf{1}_{\{1\le m\le K\}}\frac12\,\Gamma\!\Bigl(\frac12-m\Bigr)\,t^{m-1/2}
\;+\;O\!\bigl(t^{K+1/2}\bigr).
\end{equation}
In particular,
\[
\Theta_0(t)=O(t^{-1/2}),\qquad
\Theta_m(t)=\Theta_m(0)+O(t^{1/2})\ (m=1),\qquad
\Theta_m(t)=\Theta_m(0)+O(t)\ (m\ge 2),
\]
and hence $\Theta_m(t)=O(1)$ for every $m\ge 1$.
\end{proposition}

\begin{proof}
Fix $m\ge 0$. For $y>0$ and $c>0$, Mellin inversion gives
\[
\e^{-y}=\frac{1}{2\pi i}\int_{\Ree s=c}\Gamma(s)\,y^{-s}\,ds.
\]
Apply this with $y=a_n^2 t$ and choose $c>\max(0,\tfrac12-m)$ so that
$\sum_{n\ge 1}a_n^{-2(m+c)}$ converges absolutely. Fubini's theorem yields
\begin{equation}\label{eq:Theta-mellin-correct}
\Theta_m(t)=\frac{1}{2\pi i}\int_{\Ree s=c}\Gamma(s)\,t^{-s}\,S_m(s)\,ds,
\end{equation}
where
\[
S_m(s):=\sum_{n\ge 1}a_n^{-2(m+s)}=\sum_{n\ge 1}(n+\tfrac12)^{-2(m+s)}.
\]

Using Lemma~\ref{lem:hurwitz-half} with $p=2(m+s)$ and subtracting the $n=0$ term
$(\tfrac12)^{-2(m+s)}=2^{2(m+s)}$, we obtain the meromorphic identity
\begin{equation}\label{eq:Sm-correct}
S_m(s)=(2^{2(m+s)}-1)\zeta(2(m+s)) - 2^{2(m+s)}.
\end{equation}
The only pole of $\zeta$ is at $1$, hence $S_m(s)$ has at most a simple pole at
\[
2(m+s)=1\iff s=\frac12-m,
\]
and $\Gamma(s)$ has simple poles at $s=0,-1,-2,\dots$.

Fix $K\ge 0$ and shift the contour in \eqref{eq:Theta-mellin-correct} from $\Ree s=c$ to $\Ree s=s_K=-K-\tfrac12$. This crosses the poles at $s=0,-1,\dots,-K$ and possibly
$s=\frac12-m$ (which lies to the right of $s_K$ exactly when $m\le K$). All poles are simple.
By the residue theorem,
\[
\Theta_m(t)
=\sum_{\substack{s_0\in\{\frac12-m,\,0,-1,\dots,-K\}\\ s_K<\Ree s_0<c}}
\Res_{s=s_0}\bigl(\Gamma(s)t^{-s}S_m(s)\bigr)
\;+\;\frac{1}{2\pi i}\int_{\Ree s=s_K}\Gamma(s)t^{-s}S_m(s)\,ds.
\]

We compute the residues.
For $j=0,1,\dots,K$, $\Res_{s=-j}\Gamma(s)=(-1)^j/j!$, hence the residues at $s=-j$ contribute
terms $A_{m,j}t^{j}$ for suitable constants $A_{m,j}$.
At $s=\frac12-m$ (when $m\le K$), we use that $\zeta(p)$ has residue $1$ at $p=1$ and
$p=2(m+s)$ satisfies $dp/ds=2$, while $(2^p-1)$ equals $1$ at $p=1$. Therefore
$\Res_{s=\frac12-m}S_m(s)=\frac12$. Since $\Gamma$ is holomorphic at $s=\frac12-m$,
this yields the half-integer term
\[
\frac12\,\Gamma\!\Bigl(\frac12-m\Bigr)\,t^{m-1/2}.
\]
For $m=0$, this is the leading term $\frac12\Gamma(\tfrac12)t^{-1/2}=\frac{\sqrt\pi}{2}t^{-1/2}$.

Finally, we bound the shifted integral. On $\Ree s=s_K$ we have $|t^{-s}|=t^{K+1/2}$.
Moreover, by Stirling's estimate,
$|\Gamma(s_K+iy)|\ll_K (1+|y|)^{-K-1}\e^{-\pi|y|/2}$.
The function $S_m(s)$ has at most polynomial growth in $|y|$ on vertical lines by
\eqref{eq:Sm-correct} and standard bounds for $\zeta(\sigma+iy)$ in vertical strips.
Hence the integrand is absolutely integrable in $y$ and we obtain
\[
\frac{1}{2\pi i}\int_{\Ree s=s_K}\Gamma(s)t^{-s}S_m(s)\,ds=O(t^{K+1/2})
\qquad(t\to 0^+),
\]
with an implied constant depending on $K$ and $m$.
This proves \eqref{eq:Theta-K} and hence the existence of a full asymptotic expansion.
The stated in particular bounds follow by inspecting the first few exponents:
for $m\ge 1$ the $t^0$ term is present and the next exponent is $1/2$ if $m=1$
and $1$ if $m\ge 2$.
\end{proof}

\begin{remark}[Contour shifts and poles on the contour]\label{rem:contour-shift}
In Mellin inversion arguments we shift a vertical line $\Re s=c$ to the left and pick up
residues. If one of the poles of the integrand lies \emph{exactly} on the intended new
contour (e.g.\ a zeta pole at $s=s_0$ coinciding with a target line $\Re s=\sigma$),
we proceed in the standard way: either
\begin{enumerate}[label=\textup{(\alph*)},leftmargin=2.2em]
\item shift instead to $\Re s=\sigma-\eta$ for some fixed $\eta>0$, or
\item indent the contour around $s_0$ and take the principal value plus the residue.
\end{enumerate}
Since shifting slightly further left only improves decay of $t^{-s}$ as $t\to0^+$, the resulting
error term is unaffected (and in fact improves). We will not repeat this bookkeeping in each
contour shift.
\end{remark}

\subsection{Full asymptotics for $F_r(t)$}\label{subsec:Fr-asymp}

We now transfer the asymptotics of $\Theta_m(t)$ to the weighted sums $F_r(t)$.

\begin{proposition}[Full asymptotic expansion of $F_r(t)$]\label{prop:Fr-asymp}
For each integer $r\ge 0$, the function $F_r(t)$ defined in \eqref{eq:def-Fr}
admits a full asymptotic expansion as $t\to 0^+$ with exponents in $\tfrac12\Z$.

More precisely, for every integer $K\ge 0$ there exist constants $B_{r,j}$,
$j\in\{-\tfrac12,0,\tfrac12,1,\dots,K\}$, such that
\begin{equation}\label{eq:Fr-K}
F_r(t)=\sum_{j\in\{-\tfrac12,0,\tfrac12,1,\dots,K\}}B_{r,j}\,t^{j}+O(t^{K+1/2})
\qquad(t\to 0^+).
\end{equation}
\end{proposition}

\begin{proof}
Fix $r\ge 0$ and $K\ge 0$. By Proposition~\ref{prop:Fr-Theta},
\[
F_r(t)=\sum_{m=0}^\infty c_{r,m}\,\Theta_m(t),
\]
with absolute convergence for every $t>0$. For each $m\ge0$, apply
Proposition~\ref{prop:Theta-asymp} (with the same $K$) and write
\[
\Theta_m(t)=\mathcal{P}_{m,K}(t)+R_{m,K}(t),
\]
where $\mathcal{P}_{m,K}(t)$ is the truncated part from \eqref{eq:Theta-K} and
$R_{m,K}(t)=O(t^{K+1/2})$ as $t\to0^+$.

\medskip
\noindent\emph{Uniform control of the remainder and summation over $m$.}
Set $M:=K+2$. Split
\[
\sum_{m\ge0}c_{r,m}R_{m,K}(t)=\sum_{m\le M}c_{r,m}R_{m,K}(t)+\sum_{m>M}c_{r,m}R_{m,K}(t).
\]
The first sum is finite, hence $O(t^{K+1/2})$.

For the tail $m>M$, we use the remainder representation coming from the contour shift
in the proof of Proposition~\ref{prop:Theta-asymp}. Namely, if we write
$s_K:=-K-\tfrac12$, then (after subtracting residues up to order $K$)
\begin{equation}\label{eq:Rmk-integral}
R_{m,K}(t)=\frac{1}{2\pi i}\int_{\Re s=s_K}\Gamma(s)\,t^{-s}\,S_m(s)\,ds,
\end{equation}
where
\[
S_m(s)=\sum_{n\ge1}\bigl(n+\tfrac12\bigr)^{-2(m+s)}.
\]
Since $x\mapsto x^{-\alpha}$ is decreasing on $[a_1,\infty)$ for $\alpha>1$, we have the standard bound
\[
\sum_{n\ge1} a_n^{-\alpha}
\le a_1^{-\alpha}+\int_{a_1}^{\infty}x^{-\alpha}\,dx
= a_1^{-\alpha}+\frac{a_1^{-\alpha+1}}{\alpha-1}.
\]
Apply this with $\alpha:=2(m+\Re s)$. On the line $\Re s=s_K=-K-\tfrac12$ and for $m\ge K+2$,
we have $\alpha=2m-2K-1\ge3$, hence $(\alpha-1)^{-1}\le \tfrac12$. Therefore
\[
|S_m(s)|
\le \sum_{n\ge1}a_n^{-2(m+\Re s)}
\ll_K a_1^{-2(m+\Re s)}+a_1^{-2(m+\Re s)+1}
\ll_K a_1^{-2m}.
\]
Since $a_1=\tfrac32$, this gives the geometric bound $|S_m(s)|\ll_K (4/9)^m$ uniformly in $\Im s$. Since $a_1^{-2}=(\tfrac23)^2=\tfrac49=:r_0$, this yields the uniform geometric bound
\begin{equation}\label{eq:Sm-geometric}
|S_m(s)|\ll_{K} r_0^{\,m}\qquad(\Re s=s_K,\ m>M),
\end{equation}
uniformly in $\Im s$. Inserting \eqref{eq:Sm-geometric} into \eqref{eq:Rmk-integral} gives, for
$t\in(0,1]$ and $m>M$,
\[
|R_{m,K}(t)|
\le \frac{t^{-s_K}}{2\pi}\Bigl(\sup_{\Re s=s_K}|S_m(s)|\Bigr)\int_{-\infty}^{\infty}|\Gamma(s_K+iy)|\,dy
\ll_{K} r_0^{\,m}\,t^{K+1/2},
\]
since $\int_{\R}|\Gamma(s_K+iy)|\,dy<\infty$ for fixed $s_K$.
Therefore, using Lemma~\ref{lem:crm-bound} (i.e.\ $\sum_{m\ge0}|c_{r,m}|r_0^{\,m}<\infty$),
\[
\sum_{m>M}|c_{r,m}||R_{m,K}(t)| \ll_K t^{K+1/2}\sum_{m>M}|c_{r,m}|r_0^{\,m}
=O(t^{K+1/2}).
\]
Hence
\[
\sum_{m\ge0}c_{r,m}R_{m,K}(t)=O(t^{K+1/2})\qquad(t\to0^+).
\]

\medskip
\noindent\emph{Summation of the truncated pieces.}
The function $\mathcal{P}_{m,K}(t)$ is a finite linear combination of the monomials
$t^{-1/2}$, $t^0,t^1,\dots,t^K$, and (when $1\le m\le K$) the half-integer monomial $t^{m-1/2}$.
Thus $\sum_{m\ge0}c_{r,m}\mathcal{P}_{m,K}(t)$ is a finite linear combination of the same exponents,
once we justify absolute convergence for the integer-power coefficients.

For the integer powers $t^j$ ($0\le j\le K$), Proposition~\ref{prop:Theta-asymp} gives
\[
A_{m,j}=\frac{(-1)^j}{j!}\,S_m(-j).
\]
If $m\ge j+2$, then $S_m(-j)=\sum_{n\ge1}a_n^{-2(m-j)}$ converges absolutely and the same estimate as
\eqref{eq:Sm-geometric} yields
\[
|S_m(-j)|\ll_j r_0^{\,m}.
\]
Hence $\sum_{m\ge0}|c_{r,m}A_{m,j}|<\infty$ for each $j$, because only finitely many $m<j+2$ remain and
$\sum_{m\ge0}|c_{r,m}|r_0^{\,m}<\infty$.

Collecting coefficients of the finitely many exponents
$j\in\{-\tfrac12,0,\tfrac12,1,\dots,K\}$ yields constants $B_{r,j}$ such that
\eqref{eq:Fr-K} holds.
\end{proof}

\begin{lemma}[Geometric tail control for analytic weights]\label{lem:analytic-weight-tail}
Let $g(u)=\sum_{m\ge0}c_m u^m$ be analytic on $|u|<\rho$, and suppose $0\le u_n\le u_0<\rho$ for all $n$.
Then for each $M\ge0$ one has the uniform remainder bound
\[
g(u_n)=\sum_{m=0}^{M-1}c_m u_n^m \;+\; O_g\!\left(\Bigl(\frac{u_0}{\rho}\Bigr)^M\right),
\]
and consequently, for every $t\in(0,1]$,
\[
\sum_{n\ge0} g(u_n)\,e^{-a_n^2 t}
=
\sum_{m=0}^{M-1} c_m \sum_{n\ge0} u_n^m e^{-a_n^2 t}
\;+\;
O_g\!\left(\Bigl(\frac{u_0}{\rho}\Bigr)^M\sum_{n\ge0}e^{-a_n^2 t}\right),
\]
with an implied constant independent of $M$ and $t$.
\end{lemma}

\begin{remark}\label{rem:analytic-weight-tail-use}
In applications we take $g(u)=q(u)^r$ (or $g(u)=q_p(u)^r$). Since
$\sum_{n\ge0}e^{-a_n^2 t}\asymp t^{-1/2}$ as $t\to0^+$, the tail term in
Lemma~\ref{lem:analytic-weight-tail} is $O((u_0/\rho)^M t^{-1/2})$. Hence for any fixed truncation
order in $t$ one may first expand each finite $m$-sum and then let $M\to\infty$, justifying
termwise summation of the resulting asymptotic expansions.
\end{remark}

\subsection{Full heat-trace expansion and admissibility}\label{sec:final}

We now combine the reduction of Section~\ref{sec:reduction} with the asymptotics of Sections~\ref{subsec:Theta-asymp} and \ref{subsec:Fr-asymp}.

\begin{theorem}[Full small-time expansion of $\theta(t)$]\label{thm:theta-asymp}
There exist constants $(c_\nu)_{\nu\ge 0}$ and strictly increasing exponents
$j_0<j_1<\cdots\to+\infty$ with $j_\nu\in \tfrac12\Z$ such that
\[
\theta(t)\sim \sum_{\nu=0}^\infty c_\nu t^{j_\nu}\qquad (t\to 0^+).
\]
More concretely, for every integer $N\ge 0$ there exist constants
$c_{-1/2},c_0,c_{1/2},\dots,c_N$ such that
\begin{equation}\label{eq:theta-N}
\theta(t)=\sum_{j\in\{-\tfrac12,0,\tfrac12,\dots,N\}} c_j\,t^{j} + O(t^{N+1/2})
\qquad (t\to 0^+).
\end{equation}
\end{theorem}

\begin{proof}
Fix $N\ge 0$ and apply Proposition~\ref{prop:reduction} with $R=N$:
\[
\theta(t)=\sum_{r=0}^{N}\frac{(-t)^r}{r!}\,F_r(t)+O(t^{N+1})\Theta_0(t),
\qquad t\in(0,1].
\]
By Corollary~\ref{cor:theta-basic}, $\Theta_0(t)=O(t^{-1/2})$ as $t\to 0^+$, hence
\[
O(t^{N+1})\Theta_0(t)=O(t^{N+1/2}).
\]

For each fixed $r\in\{0,1,\dots,N\}$, apply Proposition~\ref{prop:Fr-asymp} with
$K=N-r$ to obtain
\[
F_r(t)=\sum_{j\in\{-\tfrac12,0,\tfrac12,\dots,N-r\}}B_{r,j}\,t^{j}
+O(t^{N-r+1/2}).
\]
Multiplying by $t^r$ gives
\[
\frac{(-t)^r}{r!}F_r(t)
=\sum_{j\in\{-\tfrac12,0,\tfrac12,\dots,N-r\}}\frac{(-1)^r}{r!}B_{r,j}\,t^{j+r}
+O(t^{N+1/2}).
\]
Summing over $r=0,1,\dots,N$ (a finite sum) and collecting like powers yields an
expansion of $\theta(t)$ involving only half-integer exponents up to $t^N$, with total
remainder $O(t^{N+1/2})$. This is exactly \eqref{eq:theta-N}. Since $N$ is arbitrary,
a full asymptotic expansion follows.
\end{proof}

\begin{corollary}[Admissibility]\label{cor:admissible}
The sequence $(\lambda_n)_{n\ge 1}$ with $\lambda_n=\tau_n^2$ is admissible
in the sense of Definition~\ref{def:admissible}.
\end{corollary}

\begin{proof}
Condition (1) holds because $\lambda_n>0$.
Condition (3) holds by Corollary~\ref{cor:theta-basic}.
Condition (2) holds because $\theta(t)$ converges absolutely for $t>0$ (Corollary~\ref{cor:theta-basic})
and has a full asymptotic expansion as $t\to 0^+$ by Theorem~\ref{thm:theta-asymp}.
\end{proof}

\section{Consequences and applications of admissibility}\label{sec:consequences}

This section records several standard consequences of admissibility which will be useful
in applications (e.g.\ zeta regularization, coefficient extraction, and large-parameter
asymptotics of canonical products). All statements are proved directly from the heat-trace expansion established in Theorem \ref{thm:theta-asymp}. The Mellin-transform representation, meromorphic continuation of the spectral zeta function, and special-value identities at nonpositive integers are instances of the general Mellin machinery in \cite[\S2]{QHS} (notably their Theorems~3--4). We include proofs for completeness and to record the precise coefficients in our normalization.

\subsection{Leading heat coefficients and a Weyl law}

We first isolate the first two coefficients in the small-time expansion of $\theta(t)$.
These coefficients are universal for the present model because the correction $q(u_n)$ is bounded.

\begin{lemma}[Two-sided bounds for $\tau_n$]\label{lem:tau-two-sided}
For every $n\ge 1$,
\[
n\le \tau_n \le n+\frac12,
\qquad\text{and hence}\qquad
n^2\le \lambda_n\le \Bigl(n+\frac12\Bigr)^2.
\]
\end{lemma}

\begin{proof}
The lower bound $\tau_n\ge n$ is Lemma \ref{lem:tau-lower}.
For the upper bound, note from \eqref{eq:tau-rho} that $\tau_n=a_n-\rho(a_n^{-1})$ with
$\rho(a_n^{-1})\ge 0$ (nonnegative coefficients), hence $\tau_n\le a_n=n+\tfrac12$.
Squaring yields the corresponding bounds for $\lambda_n=\tau_n^2$.
\end{proof}

\begin{lemma}[Model half-integer heat trace]\label{lem:Theta0-two-terms}
Let $a_n=n+\tfrac12$ and $\Theta_0(t)=\sum_{n\ge 1}\e^{-a_n^2 t}$.
Then as $t\to 0^+$,
\begin{equation}\label{eq:Theta0-two-terms}
\Theta_0(t)=\frac{\sqrt{\pi}}{2}\,t^{-1/2}-1+\frac14\,t+O(t^2).
\end{equation}
In particular, $\Theta_0(t)=\frac{\sqrt{\pi}}{2}\,t^{-1/2}-1+O(t)$.
\end{lemma}

\begin{proof}
This is the case $m=0$ of Proposition \ref{prop:Theta-asymp} with explicit residues. One has
\[
\Theta_0(t)=\frac{1}{2\pi i}\int_{\Ree s=c}\Gamma(s)\,t^{-s}\,
\Bigl((2^{2s}-1)\zeta(2s)-2^{2s}\Bigr)\,ds
\]
for any $c>\tfrac12$.
Shift the contour to $\Ree s=-\tfrac52$.
The poles crossed are:
\begin{itemize}[leftmargin=2.5em]
\item $s=\tfrac12$ (simple pole of $\zeta(2s)$ at $2s=1$);
\item $s=0,-1,-2$ (simple poles of $\Gamma(s)$).
\end{itemize}
All other possible poles lie to the left of $\Ree s=-\tfrac52$.
The remainder integral on $\Ree s=-\tfrac52$ is $O(t^{5/2})$ and hence $O(t^2)$
by the same bounds used in Proposition \ref{prop:Theta-asymp}.

We now compute residues.
Write $S_0(s):=(2^{2s}-1)\zeta(2s)-2^{2s}$.
Near $s=\tfrac12$, $\zeta(2s)\sim \frac{1}{2s-1}$, hence $S_0(s)\sim \frac{1}{2s-1}$.
Therefore
\[
\operatorname{Res}_{s=1/2}\bigl(\Gamma(s)t^{-s}S_0(s)\bigr)
=\Gamma\!\Bigl(\frac12\Bigr)\,t^{-1/2}\cdot \operatorname{Res}_{s=1/2}\frac{1}{2s-1}
=\sqrt{\pi}\,t^{-1/2}\cdot \frac12=\frac{\sqrt{\pi}}{2}\,t^{-1/2}.
\]
At $s=0$, $\Gamma(s)$ has residue $1$ and $t^{-s}=1+O(s)$, so the residue equals $S_0(0)$.
Since $\zeta(0)=-\tfrac12$ and $2^{0}-1=0$, we have $S_0(0)=-1$, giving the constant term $-1$.
At $s=-1$, $\operatorname{Res}_{s=-1}\Gamma(s)=-1$ and $t^{-s}=t$, while
\[
S_0(-1)=(2^{-2}-1)\zeta(-2)-2^{-2}=0-\frac14=-\frac14,
\]
because $\zeta(-2)=0$.
Thus the residue at $s=-1$ equals $(-1)\cdot t\cdot (-1/4)=t/4$.

Crossing $s=-2$ produces an additional residue term of size $O(t^2)$, which we absorb into the $O(t^2)$ remainder.

Summing residues gives \eqref{eq:Theta0-two-terms}.
\end{proof}

\begin{proposition}[First heat coefficients of $\theta(t)$]\label{prop:first-heat-coeffs}
As $t\to 0^+$,
\begin{equation}\label{eq:theta-first-coeffs}
\theta(t)=\frac{\sqrt{\pi}}{2}\,t^{-1/2}-1+O(t^{1/2}).
\end{equation}
In particular, in the notation of Theorem~\ref{thm:theta-asymp} one has
\[
c_{-1/2}=\frac{\sqrt{\pi}}{2},
\qquad
c_0=-1.
\]
\end{proposition}

\begin{proof}
Recall $\lambda_n=a_n^2+q(u_n)$ with $a_n=n+\tfrac12$ and $u_n=a_n^{-2}\in(0,4/9]$.
Since $q$ is continuous on $[0,4/9]$, set
\[
Q_*:=\sup_{u\in[0,4/9]}|q(u)|<\infty.
\]
Then for $t\in(0,1]$ and all $n$,
\[
\bigl|e^{-tq(u_n)}-1\bigr|
\le t\,|q(u_n)|\,e^{t|q(u_n)|}
\le t\,Q_* e^{Q_*}
=O(t),
\]
uniformly in $n$. Hence
\[
\theta(t)-\Theta_0(t)
=\sum_{n\ge1}e^{-a_n^2t}\bigl(e^{-tq(u_n)}-1\bigr)
=O(t)\sum_{n\ge1}e^{-a_n^2t}
=O(t)\,\Theta_0(t).
\]
By Lemma~\ref{lem:Theta0-two-terms}, $\Theta_0(t)=O(t^{-1/2})$, so
$\theta(t)-\Theta_0(t)=O(t^{1/2})$. Inserting the two-term expansion
$\Theta_0(t)=\frac{\sqrt{\pi}}{2}t^{-1/2}-1+O(t)$ from Lemma~\ref{lem:Theta0-two-terms}
yields \eqref{eq:theta-first-coeffs}.
\end{proof}

As a byproduct we obtain the sharp one-dimensional Weyl law for the counting function.

\begin{corollary}[Weyl law for the counting function]\label{cor:weyl}
Let
\[
N(\Lambda):=\#\{n\ge 1:\ \lambda_n\le \Lambda\},\qquad \Lambda\ge 0.
\]
Then for all $\Lambda\ge 0$,
\[
\lfloor\sqrt{\Lambda}-\frac12\ \rfloor\le\ N(\Lambda)\ \le\ \lfloor\sqrt{\Lambda}\rfloor,
\]
and in particular $N(\Lambda)=\sqrt{\Lambda}+O(1)$ as $\Lambda\to\infty$.
\end{corollary}

\begin{proof}
By Lemma \ref{lem:tau-two-sided}, $\lambda_n\le (n+\tfrac12)^2$ and $\lambda_n\ge n^2$.
If $(n+\tfrac12)^2\le \Lambda$ then $\lambda_n\le \Lambda$, hence $N(\Lambda)\ge \lfloor \sqrt{\Lambda}-\tfrac12\rfloor$.
Conversely, if $\lambda_n\le \Lambda$ then $n^2\le \Lambda$, hence $n\le \sqrt{\Lambda}$ and $N(\Lambda)\le \lfloor \sqrt{\Lambda}\rfloor$.
The displayed bounds follow.
\end{proof}

\subsection{The spectral zeta function: meromorphic continuation and special values}

Define the spectral zeta function associated with $\lambda_n$ by
\begin{equation}\label{eq:def-Z}
Z(s):=\sum_{n\ge 1}\lambda_n^{-s},\qquad \Ree s>\frac12.
\end{equation}
Absolute convergence holds for $\Ree s>\tfrac12$ because $\lambda_n\ge n^2$ (Lemma \ref{lem:tau-lower}).

\begin{lemma}[Exponential decay of the heat trace]\label{lem:theta-exp-decay}
There exists a constant $C>0$ such that for all $t\ge 1$,
\[
\theta(t)\le C\,\e^{-\lambda_1 t}.
\]
\end{lemma}

\begin{proof}
Write
\[
\theta(t)=\e^{-\lambda_1 t}\sum_{n\ge 1}\e^{-(\lambda_n-\lambda_1)t}.
\]
For $t\ge 1$, $\e^{-(\lambda_n-\lambda_1)t}\le \e^{-(\lambda_n-\lambda_1)}$, hence
\[
\theta(t)\le \e^{-\lambda_1 t}\sum_{n\ge 1}\e^{-(\lambda_n-\lambda_1)}.
\]
The series on the right converges because $\lambda_n\ge n^2$ and thus
$\e^{-(\lambda_n-\lambda_1)}\le \e^{\lambda_1}\e^{-n^2}$.
Absorb the convergent sum into the constant $C$.
\end{proof}

\begin{proposition}[Mellin transform representation of $Z(s)$]\label{prop:mellin-Z}
For $\Ree s>\tfrac12$,
\begin{equation}\label{eq:mellin-Z}
Z(s)=\frac{1}{\Gamma(s)}\int_0^\infty t^{s-1}\theta(t)\,dt.
\end{equation}
\end{proposition}

\begin{proof}
For $\Ree s>0$ and $\lambda>0$ one has
\[
\int_0^\infty t^{s-1}\e^{-\lambda t}\,dt=\lambda^{-s}\Gamma(s).
\]
For $\Ree s>\tfrac12$, the series defining $Z(s)$ converges absolutely, so by Fubini,
\[
\int_0^\infty t^{s-1}\theta(t)\,dt
=\sum_{n\ge 1}\int_0^\infty t^{s-1}\e^{-\lambda_n t}\,dt
=\Gamma(s)\sum_{n\ge 1}\lambda_n^{-s}=\Gamma(s)Z(s),
\]
which is \eqref{eq:mellin-Z}.
\end{proof}

\begin{theorem}[Meromorphic continuation of $Z(s)$ and residues]\label{thm:Z-meromorphic}
The function $Z(s)$ admits a meromorphic continuation to all of $\C$ with at most simple poles at
\[
s=\frac12-k,\qquad k=0,1,2,\dots.
\]
Moreover, if $c_{k-1/2}$ denotes the coefficient of $t^{k-1/2}$ in the heat expansion \eqref{eq:theta-N},
then
\begin{equation}\label{eq:Z-residue}
\operatorname{Res}_{s=1/2-k} Z(s)=\frac{c_{k-1/2}}{\Gamma(\tfrac12-k)}.
\end{equation}
In particular, $\operatorname{Res}_{s=1/2}Z(s)=\tfrac12$.
\end{theorem}

\begin{proof}
Fix $N\in\N$ and write the expansion \eqref{eq:theta-N} as
\[
\theta(t)=\sum_{j\in\{-\tfrac12,0,\tfrac12,\dots,N\}} c_j t^{j} + R_N(t),
\qquad R_N(t)=O(t^{N+1/2})\quad(t\to 0^+).
\]
Split the Mellin integral \eqref{eq:mellin-Z} at $1$:
\[
Z(s)=\frac{1}{\Gamma(s)}\Bigl(\int_0^1 t^{s-1}\theta(t)\,dt+\int_1^\infty t^{s-1}\theta(t)\,dt\Bigr)
=: \frac{1}{\Gamma(s)}\bigl(I_0(s)+I_\infty(s)\bigr).
\]
By Lemma \ref{lem:theta-exp-decay}, $I_\infty(s)$ converges absolutely for every $s\in\C$ and defines an entire function.

For $I_0(s)$, substitute $\theta(t)=\sum c_j t^j + R_N(t)$:
\[
I_0(s)=\sum_{j\in\{-\tfrac12,0,\tfrac12,\dots,N\}} c_j\int_0^1 t^{s+j-1}\,dt + \int_0^1 t^{s-1}R_N(t)\,dt.
\]
For $\Ree(s+j)>0$, $\int_0^1 t^{s+j-1}\,dt=\frac{1}{s+j}$, hence
\[
I_0(s)=\sum_{j\in\{-\tfrac12,0,\tfrac12,\dots,N\}}\frac{c_j}{s+j}\;+\; H_N(s),
\]
where $H_N(s)$ is holomorphic in the half-plane $\Ree s>-N-\tfrac12$ because $R_N(t)=O(t^{N+1/2})$.
Thus $Z(s)$ extends meromorphically to $\Ree s>-N-\tfrac12$ with possible poles only at $s=-j$
for $j\in\{-\tfrac12,0,\tfrac12,\dots,N\}$.
Letting $N\to\infty$ yields a meromorphic continuation to all of $\C$ with possible poles at
$s=-j$ for $j\in \{-\tfrac12,\tfrac12,\tfrac32,\dots\}$, i.e.\ $s=\tfrac12-k$ for $k\ge 0$.

To compute residues, note that $\Gamma(s)$ is holomorphic and nonzero at $s=\tfrac12-k$.
From the decomposition above, near $s=\tfrac12-k$ the only singular term in $I_0(s)$ is
$\frac{c_{k-1/2}}{s-(1/2-k)}$, so \eqref{eq:Z-residue} follows.

Finally, by Proposition \ref{prop:first-heat-coeffs}, $c_{-1/2}=\sqrt{\pi}/2$, hence
$\operatorname{Res}_{s=1/2}Z(s)=c_{-1/2}/\Gamma(1/2)=(\sqrt{\pi}/2)/\sqrt{\pi}=1/2$.
\end{proof}

\begin{remark}[Cancellation at nonpositive integers]\label{rem:gamma-cancellation}
In the Mellin representation
\[
Z(s)=\frac{1}{\Gamma(s)}\int_0^\infty t^{s-1}\theta(t)\,dt,
\]
the truncated small-$t$ expansion of $\theta(t)$ produces apparent simple poles of the integral
at points $s=-j$ coming from factors $(s+j)^{-1}$. When $j\in\mathbb{Z}_{\ge0}$, this singularity
is canceled by the simple zero of $1/\Gamma(s)$ at $s=-j$ (since $\Gamma$ has a simple pole at
each nonpositive integer). Hence $Z(s)$ is holomorphic at $s=0,-1,-2,\dots$.
\end{remark}

\begin{theorem}[Special values at nonpositive integers]\label{thm:Z-negative}
For every integer $m\ge 0$, the meromorphic continuation of $Z(s)$ is holomorphic at $s=-m$ and satisfies
\begin{equation}\label{eq:Z-negative}
Z(-m)=(-1)^m m!\,c_m,
\end{equation}
where $c_m$ is the coefficient of $t^m$ in the small-time expansion \eqref{eq:theta-N}.
In particular, $Z(0)=c_0=-1$.
\end{theorem}

\begin{proof}
Fix $m\ge 0$ and choose $N\ge m$ in the decomposition in the proof of Theorem \ref{thm:Z-meromorphic}.
Then near $s=-m$ we have
\[
Z(s)=\frac{1}{\Gamma(s)}\Bigl(\frac{c_m}{s+m} + A(s)\Bigr),
\]
where $A(s)$ is holomorphic at $s=-m$ (it contains the remaining $\frac{c_j}{s+j}$ terms with $j\ne m$,
the function $H_N(s)$, and the entire tail $I_\infty(s)$).
Since $\Gamma(s)$ has a simple pole at $s=-m$ with residue $\operatorname{Res}_{s=-m}\Gamma(s)=\frac{(-1)^m}{m!}$,
it follows that $1/\Gamma(s)$ has a simple zero at $s=-m$ and admits the expansion
\[
\frac{1}{\Gamma(s)} = (-1)^m m!\,(s+m) + O\bigl((s+m)^2\bigr).
\]
Therefore,
\[
Z(-m)=\lim_{s\to -m}\frac{1}{\Gamma(s)}\cdot \frac{c_m}{s+m}=(-1)^m m! c_m,
\]
and the remaining holomorphic term $A(s)$ vanishes in the limit because $1/\Gamma(s)\to 0$.
This proves \eqref{eq:Z-negative}.
\end{proof}

\subsection{Resolvent trace and the canonical product $W(\lambda)$}

Define the canonical genus-$0$ product
\begin{equation}\label{eq:def-W}
W(\lambda):=\prod_{n\ge 1}\Bigl(1+\frac{\lambda}{\lambda_n}\Bigr).
\end{equation}
Since $\sum_{n\ge 1}\lambda_n^{-1}<\infty$ (by $\lambda_n\ge n^2$), the product converges
absolutely and locally uniformly, hence $W$ is entire.

\begin{proposition}[Logarithmic derivative of $W$]\label{prop:log-derivative-W}
For every $\lambda\in\C\setminus\{-\lambda_n:\ n\ge 1\}$,
\begin{equation}\label{eq:Wprime-over-W}
\frac{W'(\lambda)}{W(\lambda)}=\sum_{n\ge 1}\frac{1}{\lambda+\lambda_n},
\end{equation}
where the series converges locally uniformly in $\lambda$.
\end{proposition}

\begin{proof}
Fix a compact set $K\subset \C\setminus\{-\lambda_n\}$.
Then there exists $\delta>0$ such that $|\lambda+\lambda_n|\ge \delta\,\lambda_n$ for all $\lambda\in K$ and all $n$.
Consequently,
\[
\sum_{n\ge 1}\sup_{\lambda\in K}\Bigl|\frac{1}{\lambda+\lambda_n}\Bigr|
\le \delta^{-1}\sum_{n\ge 1}\lambda_n^{-1}<\infty.
\]
Thus $\sum_{n\ge 1}\log(1+\lambda/\lambda_n)$ converges uniformly on $K$ and may be differentiated termwise,
yielding \eqref{eq:Wprime-over-W}.
\end{proof}

Define the resolvent trace
\[
R(\lambda):=\sum_{n\ge 1}\frac{1}{\lambda+\lambda_n}.
\]
For $\Ree\lambda>0$ it admits a Laplace transform representation.

\begin{lemma}[Laplace representation]\label{lem:laplace-resolvent}
For $\Ree\lambda>0$,
\begin{equation}\label{eq:laplace-resolvent}
R(\lambda)=\int_0^\infty \e^{-\lambda t}\theta(t)\,dt.
\end{equation}
\end{lemma}

\begin{proof}
For $\Ree\lambda>0$ and $\lambda_n>0$,
\[
\int_0^\infty \e^{-(\lambda+\lambda_n)t}\,dt=\frac{1}{\lambda+\lambda_n}.
\]
Moreover, for each fixed $\lambda$ with $\Ree\lambda>0$, the series
$\sum_{n\ge 1} \frac{1}{|\lambda+\lambda_n|}$ converges because it is dominated by a constant multiple of
$\sum \lambda_n^{-1}$. Hence by Fubini,
\[
\int_0^\infty \e^{-\lambda t}\theta(t)\,dt
=\sum_{n\ge 1}\int_0^\infty \e^{-(\lambda+\lambda_n)t}\,dt
=\sum_{n\ge 1}\frac{1}{\lambda+\lambda_n}=R(\lambda).
\]
\end{proof}

We now derive large-$\lambda$ asymptotics for $R(\lambda)$ and $\log W(\lambda)$ from the small-time heat expansion. This is a classical Watson-lemma mechanism, we include a proof tailored to our situation.

\begin{remark}[Branch conventions in sectors]\label{rem:sector-branches}
Fix $\delta\in(0,\tfrac{\pi}{2})$ and set
\[
\Sigma_\delta:=\{\lambda\in\mathbb{C}:\ |\arg\lambda|\le \tfrac{\pi}{2}-\delta\}.
\]
On $\Sigma_\delta$ we use the principal branch of $\log\lambda$ (so $\arg\lambda\in(-\pi,\pi)$)
and define $\lambda^\alpha:=\exp(\alpha\log\lambda)$. Since $\Sigma_\delta$ is simply connected and
contained in the open right half-plane, these branches are single-valued and holomorphic on
$\Sigma_\delta$. In particular, the straight segment from $1$ to $\lambda$ lies in $\Sigma_\delta$,
so term-by-term integration of asymptotic series in $\lambda$ along that segment is legitimate.
\end{remark}

\begin{lemma}[Watson-type lemma for Laplace transforms]\label{lem:watson}
Let $f:(0,\infty)\to\C$ be measurable and assume:
\begin{enumerate}[label=\textup{(\alph*)},leftmargin=2.5em]
\item as $t\to 0^+$,
\[
f(t)=\sum_{\nu=0}^{M} a_\nu t^{\alpha_\nu}+O(t^{\alpha_{M+1}}),
\qquad -1<\alpha_0<\alpha_1<\cdots<\alpha_{M+1};
\]
\item there exist $c,C>0$ such that $|f(t)|\le C\e^{-c t}$ for all $t\ge 1$.
\end{enumerate}
Then for every $\delta\in(0,\tfrac{\pi}{2})$, as $|\lambda|\to\infty$ with $|\arg\lambda|\le \tfrac{\pi}{2}-\delta$,
\[
\int_0^\infty \e^{-\lambda t}f(t)\,dt
=\sum_{\nu=0}^{M} a_\nu \Gamma(\alpha_\nu+1)\lambda^{-\alpha_\nu-1}
+O(|\lambda|^{-\alpha_{M+1}-1}).
\]
\end{lemma}

\begin{proof}
Fix $\delta\in(0,\tfrac{\pi}{2})$ and set $\Sigma_\delta:=\{\lambda:\ |\arg\lambda|\le \tfrac{\pi}{2}-\delta\}$.
For $\lambda\in\Sigma_\delta$, $\Ree\lambda\ge |\lambda|\sin\delta$.

Split the Laplace integral at $1$:
\[
\int_0^\infty \e^{-\lambda t}f(t)\,dt=\int_0^1 \e^{-\lambda t}f(t)\,dt+\int_1^\infty \e^{-\lambda t}f(t)\,dt=:I_0(\lambda)+I_\infty(\lambda).
\]
By assumption (b), for $\lambda\in\Sigma_\delta$,
\[
|I_\infty(\lambda)|\le \int_1^\infty \e^{-\Ree\lambda\, t}\,C\e^{-ct}\,dt
\le C\int_1^\infty \e^{-(|\lambda|\sin\delta+c)t}\,dt
\ll_\delta \e^{-|\lambda|\sin\delta},
\]
hence $I_\infty(\lambda)$ is exponentially small and in particular $O(|\lambda|^{-N})$ for every $N$.

For $I_0(\lambda)$, substitute the expansion from (a):
\[
I_0(\lambda)=\sum_{\nu=0}^{M} a_\nu \int_0^1 \e^{-\lambda t}t^{\alpha_\nu}\,dt + \int_0^1 \e^{-\lambda t}O(t^{\alpha_{M+1}})\,dt.
\]
The remainder integral satisfies
\begin{align*}
    \Bigl|\int_0^1 \e^{-\lambda t}O(t^{\alpha_{M+1}})\,dt\Bigr|
&\ll \int_0^1 \e^{-\Ree\lambda\, t} t^{\alpha_{M+1}}\,dt\\
&\le \int_0^\infty \e^{-(|\lambda|\sin\delta) t} t^{\alpha_{M+1}}\,dt\\
&=\Gamma(\alpha_{M+1}+1)\,(|\lambda|\sin\delta)^{-\alpha_{M+1}-1},
\end{align*}

so it is $O_\delta(|\lambda|^{-\alpha_{M+1}-1})$.

Finally, for each $\alpha>-1$ and $\lambda\in\Sigma_\delta$,
\[
\int_0^1 \e^{-\lambda t}t^{\alpha}\,dt
=\lambda^{-\alpha-1}\int_0^{\lambda} \e^{-u}u^{\alpha}\,du
=\Gamma(\alpha+1)\lambda^{-\alpha-1}+O_\delta(\e^{-|\lambda|\sin\delta}),
\]
because $\int_\lambda^\infty \e^{-u}u^{\alpha}\,du=O(\e^{-\Ree\lambda})$.
Combining the above completes the proof.
\end{proof}

\begin{theorem}[Asymptotics of the resolvent trace]\label{thm:resolvent-asymp}
Fix $\delta\in(0,\tfrac{\pi}{2})$.
As $|\lambda|\to\infty$ with $|\arg\lambda|\le \tfrac{\pi}{2}-\delta$,
\begin{equation}\label{eq:resolvent-asymp}
R(\lambda)=\sum_{j\in\{-\tfrac12,0,\tfrac12,\dots,N\}} c_j \Gamma(j+1)\lambda^{-j-1}+O(|\lambda|^{-N-3/2})
\end{equation}
for every $N\in\N$.
In particular the leading term is $R(\lambda)\sim \frac{\pi}{2}\lambda^{-1/2}$.
\end{theorem}

\begin{proof}
Apply Lemma \ref{lem:watson} to $f(t)=\theta(t)$, using Theorem \ref{thm:theta-asymp} for condition (a)
and Lemma \ref{lem:theta-exp-decay} for condition (b), and then use \eqref{eq:laplace-resolvent}.
The leading coefficient follows from Proposition \ref{prop:first-heat-coeffs}:
$c_{-1/2}\Gamma(1/2)=\frac{\sqrt{\pi}}{2}\cdot \sqrt{\pi}=\frac{\pi}{2}$.
\end{proof}

We integrate the resolvent expansion to obtain the large-$\lambda$ expansion of $\log W(\lambda)$.
Fix $\delta\in(0,\tfrac{\pi}{2})$ and define a branch of $\log W(\lambda)$ on $\Sigma_\delta$ by analytic continuation
from $\lambda=0$ along any path in $\Sigma_\delta$.

\begin{theorem}[Asymptotics of the canonical product]\label{thm:logW-asymp}
Fix $\delta\in(0,\tfrac{\pi}{2})$ and let $\Sigma_\delta=\{\lambda:\ |\arg\lambda|\le \tfrac{\pi}{2}-\delta\}$.
There exists a constant $C_W\in\C$ (depending on the chosen branch of $\log W$) such that, as $|\lambda|\to\infty$
with $\lambda\in\Sigma_\delta$,
\begin{equation}\label{eq:logW-asymp}
\log W(\lambda)
=
\pi \lambda^{1/2} + c_0\log\lambda + C_W
-\sum_{\substack{j\in\{\tfrac12,1,\tfrac32,\dots\}\\ j\le N}}\frac{c_j\Gamma(j+1)}{j}\,\lambda^{-j}
\;+\; O(|\lambda|^{-N-1/2}),
\end{equation}
for every $N\in\N$, where $\lambda^{1/2}$ and $\log\lambda$ are taken with their principal branches on $\Sigma_\delta$.
In particular, for every integer $m\ge 1$, the coefficient of $\lambda^{-m}$ in the asymptotic expansion of $\log W(\lambda)$ equals
\begin{equation}\label{eq:logW-integer-coeff}
[\lambda^{-m}]\,\log W(\lambda)=\frac{(-1)^{m+1}}{m}\,Z(-m).
\end{equation}
\end{theorem}

\begin{proof}
By Proposition \ref{prop:log-derivative-W}, $d(\log W)/d\lambda=R(\lambda)$ on $\Sigma_\delta$.
Integrate $R(\lambda)$ along the straight segment from $1$ to $\lambda$.
Using \eqref{eq:resolvent-asymp}, term-by-term integration gives:
\begin{itemize}[leftmargin=2.5em]
\item the $j=-\tfrac12$ term integrates to $2c_{-1/2}\Gamma(1/2)\lambda^{1/2}=\pi\lambda^{1/2}$;
\item the $j=0$ term integrates to $c_0\log\lambda$;
\item each $j>0$ term integrates to $-\frac{c_j\Gamma(j+1)}{j}\lambda^{-j}$.
\end{itemize}
The integrated remainder $O(|\lambda|^{-N-3/2})$ contributes $O(|\lambda|^{-N-1/2})$.
All contributions from the lower limit $1$ and from the choice of branch are absorbed into the constant $C_W$.
This proves \eqref{eq:logW-asymp}.

For \eqref{eq:logW-integer-coeff}, take $j=m\in\N$ in \eqref{eq:logW-asymp}:
the coefficient of $\lambda^{-m}$ equals $-\frac{c_m\Gamma(m+1)}{m}=-\frac{m!}{m}c_m$.
By Theorem \ref{thm:Z-negative}, $c_m=(-1)^m Z(-m)/m!$, hence the coefficient equals
$\frac{(-1)^{m+1}}{m}Z(-m)$.
\end{proof}

Recall that
\[
\varphi(w)=\prod_{n\ge 1}\Bigl(1-\frac{w^2}{\tau_n^2}\Bigr)=\prod_{n\ge 1}\Bigl(1+\frac{-w^2}{\lambda_n}\Bigr)=W(-w^2).
\]
The previous theorem immediately yields the coefficient extraction statement needed for coefficient-comparison arguments.

\begin{corollary}[Coefficient extraction at infinity]\label{cor:logphi-coeff}
Fix a sector $\mathcal{S}$ in the $w$-plane such that $-w^2\in\Sigma_\delta$ for some $\delta\in(0,\tfrac{\pi}{2})$
(e.g.\ $|\arg w+\tfrac{\pi}{2}|\le \tfrac{\pi}{4}-\varepsilon$ for some $\varepsilon>0$).
Then as $|w|\to\infty$ with $w\in\mathcal{S}$,
\begin{align*}
\log\varphi(w)=\log W(-w^2)
&= \textup{(terms involving $w$ and odd powers of $w^{-1}$)}\\
&-\sum_{m=1}^{N}\frac{Z(-m)}{m}\,w^{-2m} \;+\;O(|w|^{-2N-1})    
\end{align*}

for every $N\in\N$.
In particular, for every $m\ge 1$,
\begin{equation}\label{eq:logphi-evencoeff}
[w^{-2m}]\,\log\varphi(w)=-\frac{Z(-m)}{m}.
\end{equation}
\end{corollary}

\begin{proof}
Apply Theorem \ref{thm:logW-asymp} with $\lambda=-w^2$.
The terms $\lambda^{1/2}$ and $\lambda^{-j}$ with $j\in \tfrac12+\Z_{\ge 0}$ become
odd powers of $w$ and $w^{-1}$ after choosing a consistent branch of $\sqrt{\lambda}$ on $\Sigma_\delta$.
For integer $m\ge 1$, $\lambda^{-m}=(-w^2)^{-m}=(-1)^m w^{-2m}$, and the coefficient from
\eqref{eq:logW-integer-coeff} yields $-\frac{Z(-m)}{m}w^{-2m}$.
The remainder $O(|\lambda|^{-N-1/2})$ becomes $O(|w|^{-2N-1})$.
\end{proof}

\subsection{A robust admissibility criterion: analytic perturbations of quadratic lattices}

The proof given above is stable under perturbations and depends only on analyticity of the correction term.
We record a general criterion which may be useful in related extremal-function problems.

\begin{theorem}[General admissibility criterion]\label{thm:general-criterion}
Fix $a>0$. Assume that $q$ is analytic in a neighborhood of the closed interval $[0,a^{-2}]$
(equivalently, there exists $\rho>a^{-2}$ such that $q$ is analytic on $|u|<\rho$), and that
$q$ is real-valued on $[0,a^{-2}]$. Define
\[
\lambda_n:=(n+a)^2+q\bigl((n+a)^{-2}\bigr),\qquad n\ge 0,
\]
and assume $\lambda_n>0$ for all $n$. Then the sequence $(\lambda_n)_{n\ge 0}$ is admissible in the
sense of Definition~\ref{def:admissible}.
\end{theorem}

\begin{proof}
Set $a_n:=n+a$ and $u_n:=a_n^{-2}$. Then $u_n\in(0,u_0]$ where $u_0=a^{-2}$.
Choose $\rho>a^{-2}$ such that $q$ is analytic on the disk $|u|<\rho$.
Set $u_0:=a^{-2}$ so that $0<u_n\le u_0<\rho$ for all $n$.
Then $q$ is bounded on $[0,u_0]$, say $\sup_{u\in[0,u_0]}|q(u)|\le Q_*<\infty$.

\emph{(1) and absolute convergence.}
Since $q(u_n)$ is real and bounded and $a_n^2\to\infty$, we have $\lambda_n\to+\infty$ and
$\Re\lambda_n=\lambda_n>0$. Moreover, $\lambda_n\ge a_n^2-Q_*\gg n^2$, hence
$\sum_{n\ge0}e^{-\lambda_nt}$ converges absolutely for every $t>0$ and
$\sum_{n\ge0}e^{-\lambda_nt}\ll t^{-1/2}$ as $t\to0^+$, implying admissibility~(3) as before.

\emph{(2) Full heat-trace expansion.}
Write
\[
\theta(t):=\sum_{n\ge0}e^{-\lambda_nt}
=\sum_{n\ge0}e^{-a_n^2t}e^{-tq(u_n)}.
\]
Fix $N\in\N$. Since $|q(u_n)|\le Q_*$, Taylor's theorem gives for $t\in(0,1]$,
uniformly in $n$,
\[
e^{-tq(u_n)}=\sum_{r=0}^{N}\frac{(-t)^r}{r!}\,q(u_n)^r+O(t^{N+1}).
\]
Hence
\[
\theta(t)=\sum_{r=0}^{N}\frac{(-t)^r}{r!}\,F_r(t)+O(t^{N+1})\Theta_{a,0}(t),
\qquad
F_r(t):=\sum_{n\ge0}q(u_n)^r e^{-a_n^2t},
\]
where $\Theta_{a,0}(t):=\sum_{n\ge0}e^{-a_n^2t}$.
Since $\Theta_{a,0}(t)=O(t^{-1/2})$ as $t\to0^+$, the remainder is $O(t^{N+1/2})$.

It remains to show that each $F_r(t)$ has a full expansion.
Because $q$ is analytic on $\{|u|\le r\}$, the function $q(u)^r$ has a power series
\[
q(u)^r=\sum_{m\ge0}c_{r,m}u^m,\qquad |u|<r,
\]
and by Cauchy's estimate there is $C_r>0$ such that $|c_{r,m}|\le C_r r^{-m}$.
Since $0<u_n\le u_0<r$,
\[
\sum_{m\ge0}|c_{r,m}|\sum_{n\ge0}u_n^m e^{-a_n^2t}
\le \Theta_{a,0}(t)\sum_{m\ge0}C_r\Bigl(\frac{u_0}{r}\Bigr)^m<\infty,
\]
so we may exchange sums and obtain
\[
F_r(t)=\sum_{m\ge0}c_{r,m}\,\Theta_{a,m}(t),
\qquad
\Theta_{a,m}(t):=\sum_{n\ge0}a_n^{-2m}e^{-a_n^2t}.
\]

Finally, each $\Theta_{a,m}(t)$ has a full small-time expansion by the Mellin/Hurwitz zeta
contour shift (exactly as in your Proposition~\ref{prop:Theta-asymp}, with $\zeta(\cdot,a)$),
and the $m$-sum above converges absolutely and geometrically, so the expansions may be summed termwise.
This yields a full expansion for each $F_r(t)$ and hence for $\theta(t)$.
Therefore admissibility~(2) holds, completing the proof.
\end{proof}

\section{Proof of Conjecture~1 in \cite{BOCRS}}\label{sec:conj1}

In this section we prove Conjecture~1 of \cite{BOCRS}. Recall that
\[
L_+(s):=\sum_{n\ge1}\tau_n^{-s}\qquad(\Ree s>1),
\]
and that $L_+$ admits a meromorphic continuation to $\C$ (e.g.\ via the spectral zeta function $Z(s)=\sum_{n\ge1}\lambda_n^{-s}=L_+(2s)$ from Theorem~\ref{thm:Z-meromorphic}).

\begin{theorem}[Conjecture~1 of \cite{BOCRS}]\label{thm:conj1}
For every $k\in\Z$ we have
\[
L_+(-2k)=\frac{L_+(2k)}{(2\pi i C)^{2k}}.
\]
\end{theorem}

\subsection{The inversion identity for $\varphi$}

We adopt the notation of \cite{BOCRS}. In particular,
\[
\varphi(z)=\Phi(z)\Phi(-z),\qquad
\Phi(z)=\prod_{n\ge1}\Bigl(1+(-1)^n\frac{z}{\tau_n}\Bigr),
\]
and $\varphi$ is the extremal function normalized by $\varphi(0)=1$ (hence $\Phi(0)=1$).
Let
\begin{equation}\label{eq:def-T}
Tz:=\frac{1}{2\pi i C z},\qquad z\in\C^\ast,
\end{equation}
and define the gauged function
\begin{equation}\label{eq:def-A}
A(z):=\e^{\frac{1}{4Cz}}\Phi(z),\qquad z\in\C^\ast.
\end{equation}
Note that $A(z)A(-z)=\Phi(z)\Phi(-z)=\varphi(z)$ for all $z\neq 0$.

\begin{lemma}[Gauged functional equation]\label{lem:gauged-FE}
Let $A$ be defined by \eqref{eq:def-A}. Then for every $z\in\C^\ast$,
\begin{equation}\label{eq:gauged-FE}
2\sqrt{\pi C}\,z\,A(z)=\e^{-i\pi/4}A(Tz)+\e^{i\pi/4}A(-Tz).
\end{equation}
\end{lemma}

\begin{proof}
We start from the functional equation \cite[(1.6)]{BOCRS} for $\Phi$ (valid for every $z\in\C^\ast$),
\[
\Phi(z)\e^{\frac{1}{4Cz}}
=
\frac{
\e^{-i\pi/4+\frac{i\pi}{2}z}\,\Phi(Tz)
+
\e^{i\pi/4-\frac{i\pi}{2}z}\,\Phi(-Tz)
}{2\sqrt{\pi C}\,z}.
\]
Using \eqref{eq:def-A} and the identity
\[
\frac{1}{4C(Tz)}=\frac{1}{4C}\cdot (2\pi i C z)=\frac{\pi i}{2}z,
\]
we have $\Phi(Tz)=\e^{-\frac{1}{4C(Tz)}}A(Tz)=\e^{-\frac{\pi i}{2}z}A(Tz)$ and similarly
$\Phi(-Tz)=\e^{\frac{\pi i}{2}z}A(-Tz)$. Substituting these expressions into \cite[(1.6)]{BOCRS}
cancels the $z$--dependent exponentials and yields \eqref{eq:gauged-FE}.
\end{proof}

\begin{proposition}[Inversion identity]\label{prop:inversion-identity}
For every $z\in\C^\ast$,
\begin{equation}\label{eq:phi-inversion}
\varphi(Tz)=\pi C z^2\bigl(A(z)^2+A(-z)^2\bigr).
\end{equation}
Equivalently, with $R(z):=A(z)/A(-z)$,
\begin{equation}\label{eq:phi-inversion-R}
\varphi(Tz)=\pi C z^2\,\varphi(z)\,\bigl(R(z)+R(z)^{-1}\bigr).
\end{equation}
\end{proposition}

\begin{proof}
Apply Lemma~\ref{lem:gauged-FE} at $z$ and at $-z$. Since $T(-z)=-Tz$, we obtain the linear system
\begin{align*}
\e^{-i\pi/4}A(Tz)+\e^{i\pi/4}A(-Tz) &= 2\sqrt{\pi C}\,z\,A(z),\\
\e^{i\pi/4}A(Tz)+\e^{-i\pi/4}A(-Tz) &= -\,2\sqrt{\pi C}\,z\,A(-z).
\end{align*}
Solving for $A(Tz)$ and $A(-Tz)$ and multiplying the two solutions, the mixed term cancels because
$\e^{-i\pi/2}+\e^{i\pi/2}=0$, and we get
\[
A(Tz)A(-Tz)=\pi C z^2\bigl(A(z)^2+A(-z)^2\bigr).
\]
Finally, observe that $A(Tz)A(-Tz)=\Phi(Tz)\Phi(-Tz)=\varphi(Tz)$ (the gauge factors cancel),
which proves \eqref{eq:phi-inversion}. The equivalent form \eqref{eq:phi-inversion-R} follows from
$A(z)^2+A(-z)^2=A(z)A(-z)\bigl(R+R^{-1}\bigr)=\varphi(z)\bigl(R+R^{-1}\bigr)$.
\end{proof}

\subsection{Sectorial logarithms and an even-coefficient invariance}

We now extract from \eqref{eq:phi-inversion-R} a parity statement about the \emph{even} coefficients
in the small-$z$ expansion of $\log\varphi(Tz)$.


\begin{definition}[Poincar\'e asymptotic coefficients with logarithms]\label{def:poincare-coeff}
Let $S$ be a sector with vertex at $0$ and fix a holomorphic branch of $\log z$ on $S$.
Let $F$ be holomorphic on $S\cap\{0<|z|<r\}$.

We say that $F$ admits a (logarithmic) Poincar\'e asymptotic expansion on $S$ if there exist
an integer $n_0$ and $P\in\mathbb{N}_0$ such that for every integer $N>n_0$ one has
\[
F(z)=\sum_{p=0}^{P}(\log z)^p \sum_{n=n_0}^{N-1} a_{n,p}\,z^n \;+\; O(|z|^{N})
\qquad (z\to 0,\ z\in S),
\]
for some coefficients $a_{n,p}\in\mathbb{C}$ (depending only on $F$ and $S$).
These coefficients are uniquely determined, we write
\[
[z^n(\log z)^p]_{\mathrm{as}}\,F:=a_{n,p},
\qquad\text{and in particular}\qquad
[z^n]_{\mathrm{as}}\,F:=[z^n(\log z)^0]_{\mathrm{as}}\,F.
\]
If $F$ is holomorphic at $0$ (so that no logarithmic or negative-power terms occur),
then $[z^n]_{\mathrm{as}}F$ coincides with the usual Taylor coefficient, denoted $[z^n]_{\mathrm{Tay}}F$.
\end{definition}

\begin{remark}\label{rem:as-infty}
If $G$ is holomorphic in a sector at $\infty$ and admits an asymptotic expansion in inverse powers,
we interpret $[w^{-n}(\log w)^p]_{\mathrm{as}}G$ as $[z^{n}(\log z)^p]_{\mathrm{as}}(G(1/z))$ as $z\to 0$
in the corresponding sector.
\end{remark}

\begin{lemma}[Odd Laurent structure of $\log R$ near $0$]\label{lem:logR-odd}
There exists $r>0$ such that on $|z|<r$ one can choose a holomorphic branch of
\[
H(z):=\log R(z)=\log\frac{A(z)}{A(-z)}
\]
satisfying $H(z)=\frac{1}{2Cz}+h(z)$, where $h$ is holomorphic at $0$ and odd
(hence $h(z)=\sum_{m\ge0} h_{2m+1}z^{2m+1}$).
\end{lemma}

\begin{proof}
Using \eqref{eq:def-A} we have
\[
R(z)=\frac{A(z)}{A(-z)}=\e^{\frac{1}{2Cz}}\frac{\Phi(z)}{\Phi(-z)}.
\]
Since $\Phi(0)=1$, the ratio $\Phi(z)/\Phi(-z)$ is holomorphic near $0$ and equals $1$ at $z=0$.
Hence we may choose the holomorphic logarithm
\[
h(z):=\log\frac{\Phi(z)}{\Phi(-z)}
\]
with $h(0)=0$. Then $h(-z)=\log\frac{\Phi(-z)}{\Phi(z)}=-h(z)$, so $h$ is odd.
Defining $H(z):=\frac{1}{2Cz}+h(z)$ gives the desired branch of $\log R(z)$.
\end{proof}

We then \emph{define} $\log\!\bigl(R(z)+R(z)^{-1}\bigr)$ on $S_\varepsilon\cap\{|z|<r\}$ by the identity $\log\!\bigl(R+R^{-1}\bigr)=H+\log\!\bigl(1+e^{-2H}\bigr)$, where $\log(1+u)$ is the principal holomorphic logarithm on $|u|<1/2$, this choice is compatible with the factorization $R+R^{-1}=e^{H}(1+e^{-2H})$ and avoids branch ambiguities.

\begin{lemma}[Even-coefficient invariance]\label{lem:even-coeff-invariance}
Fix $\varepsilon\in(0,\pi/4)$ and let
\[
S_\varepsilon:=\Bigl\{z\in\C:\ 0<|z|<\tau_1,\ |\arg z|<\varepsilon\Bigr\}.
\]
On $S_\varepsilon$ we take the branch of $\log z$ with $\arg z\in(-\varepsilon,\varepsilon)$.
Then $\log\varphi(Tz)$ admits a logarithmic Poincar\'e asymptotic expansion as $z\to0$ in $S_\varepsilon$
in the sense of Definition~\ref{def:poincare-coeff}, and for every integer $k\ge1$,
\begin{equation}\label{eq:even-coeff-invariance}
[z^{2k}]_{\mathrm{as}}\log\varphi(Tz)=[z^{2k}]_{\mathrm{Tay}}\log\varphi(z).
\end{equation}
\end{lemma}

\begin{proof}
On $S_\varepsilon$ we have $\Re(1/z)\ge (\cos\varepsilon)/|z|$. By Lemma~\ref{lem:logR-odd},
we may choose $r>0$ and a holomorphic branch
\[
H(z):=\log R(z)=\log\frac{A(z)}{A(-z)}
\qquad (0<|z|<r),
\]
such that $H(z)=\frac{1}{2Cz}+h(z)$ with $h$ holomorphic at $0$ and odd.

Hence, as $z\to0$ in $S_\varepsilon$,
\[
\Re H(z)=\Re\Bigl(\frac{1}{2Cz}\Bigr)+O(|z|)\to +\infty,
\]
so $\e^{-2H(z)}=O(\e^{-c/|z|})$ for some $c>0$, and therefore for every $N\in\N$,
\begin{equation}\label{eq:flat}
\e^{-2H(z)}=O(|z|^N)\qquad (z\to0,\ z\in S_\varepsilon).
\end{equation}
Shrinking $r$ if necessary, we may assume $|\e^{-2H(z)}|\le 1/2$ on $S_\varepsilon\cap\{|z|<r\}$,
so $\log(1+\e^{-2H(z)})$ is holomorphic there and is flat at $0$ by \eqref{eq:flat}.

Using \eqref{eq:phi-inversion-R} and writing
\[
R(z)+R(z)^{-1}=\e^{H(z)}\bigl(1+\e^{-2H(z)}\bigr),
\]
we obtain on $S_\varepsilon\cap\{|z|<r\}$ the identity
\begin{equation}\label{eq:logphiT}
\log\varphi(Tz)
=
\log(\pi C z^2)+\log\varphi(z)+H(z)+\log\bigl(1+\e^{-2H(z)}\bigr),
\end{equation}
where $\log(\pi C z^2)=\log(\pi C)+2\log z$ uses our fixed branch of $\log z$ on $S_\varepsilon$.
Since $\varphi(0)=1$ and $\varphi$ has no zeros in $|z|<\tau_1$, $\log\varphi(z)$ is holomorphic in $|z|<\tau_1$
and has a Taylor expansion with \emph{only even powers}.

Now inspect the contributions to the coefficient $[z^{2k}]_{\mathrm{as}}$ for $k\ge1$:
\begin{itemize}[leftmargin=2.5em]
\item $\log(\pi C z^2)=\log(\pi C)+2\log z$ contributes no positive powers of $z$.
\item $H(z)=\frac{1}{2Cz}+h(z)$ contributes no positive \emph{even} powers of $z$ because
$\frac{1}{2Cz}$ is a $z^{-1}$ term and $h$ is odd.
\item $\log(1+\e^{-2H(z)})$ is flat by \eqref{eq:flat} and thus contributes no terms to the (logarithmic)
Poincar\'e expansion.
\end{itemize}
Therefore, for every $k\ge1$, the coefficient of $z^{2k}$ in the logarithmic Poincar\'e expansion of
$\log\varphi(Tz)$ equals the Taylor coefficient of $z^{2k}$ in $\log\varphi(z)$, which is precisely
\eqref{eq:even-coeff-invariance}.
\end{proof}

\subsection{Coefficient identifications at $0$ and $\infty$}

We now compute the even coefficients on the right- and left-hand side of
\eqref{eq:even-coeff-invariance}.

\begin{lemma}[Taylor coefficients at $0$]\label{lem:tay-at-zero}
For $|z|<\tau_1$ one has
\begin{equation}\label{eq:logphi-taylor}
\log\varphi(z)=-\sum_{k\ge1}\frac{L_+(2k)}{k}\,z^{2k}.
\end{equation}
In particular, for every $k\ge1$,
\[
[z^{2k}]_{\mathrm{Tay}}\log\varphi(z)=-\frac{L_+(2k)}{k}.
\]
\end{lemma}

\begin{proof}
From \cite[(1.1) and (1.5)]{BOCRS} (or directly from the product for $\varphi$) we have
\[
\varphi(z)=\prod_{n\ge1}\Bigl(1-\frac{z^2}{\tau_n^2}\Bigr),
\]
and $\sum_{n\ge1}\tau_n^{-2}<\infty$ implies the product converges locally uniformly.
For $|z|<\tau_1$ we may expand $\log(1-w)=-\sum_{k\ge1} w^k/k$ and sum termwise to obtain
\[
\log\varphi(z)=\sum_{n\ge1}\log\Bigl(1-\frac{z^2}{\tau_n^2}\Bigr)
=-\sum_{k\ge1}\frac{z^{2k}}{k}\sum_{n\ge1}\tau_n^{-2k},
\]
which is \eqref{eq:logphi-taylor}.
\end{proof}

\begin{lemma}[Asymptotic coefficients at $\infty$]\label{lem:as-at-infty}
Let $w\to\infty$ in a sector $\mathcal{S}$ such that $\lambda=-w^2$ stays in a sector $\Sigma_\delta$
as in Corollary~\ref{cor:logphi-coeff}. Then for every $N\in\N$,
\begin{equation}\label{eq:logphi-infty}
\log\varphi(w)=\textup{(terms involving $w$ and odd powers of $w^{-1}$)}\;-\;\sum_{m=1}^{N}\frac{L_+(-2m)}{m}\,w^{-2m}
\;+\;O(|w|^{-2N-1}).
\end{equation}
Consequently, for each $m\ge1$,
\[
[w^{-2m}]_{\mathrm{as}}\log\varphi(w)=-\frac{L_+(-2m)}{m}.
\]
\end{lemma}

\begin{proof}
This is precisely Corollary~\ref{cor:logphi-coeff}, noting that $Z(-m)=L_+(2(-m))=L_+(-2m)$.
\end{proof}

\begin{lemma}[Asymptotic coefficients of $\log\varphi(Tz)$ near $0$]\label{lem:as-Tz}
Let $\varepsilon\in(0,\pi/4)$ and $S_\varepsilon$ be as in Lemma~\ref{lem:even-coeff-invariance}.
Then for each integer $k\ge1$,
\begin{equation}\label{eq:coeff-logphiT}
[z^{2k}]_{\mathrm{as}}\log\varphi(Tz)=-\frac{L_+(-2k)}{k}\,(2\pi i C)^{2k}.
\end{equation}
\end{lemma}

\begin{proof}
We interpret asymptotic coefficients in the sense of Remark~\ref{rem:as-infty}. Logarithmic terms (coming from $\log(-w^2)$) only affect the constant term in $z$ and hence do not contribute to $[z^{2k}]_{\mathrm{as}}$ for $k\ge1$.

For $z\in S_\varepsilon$, we have $w=Tz=\frac{1}{2\pi i C z}$ with $|w|\to\infty$ and
$|\arg w+\frac{\pi}{2}|<\varepsilon$. Hence $-w^2$ lies in a sector $\Sigma_\delta$ around the positive real axis,
with $\delta=\frac{\pi}{2}-2\varepsilon>0$, so Lemma~\ref{lem:as-at-infty} applies to $\log\varphi(w)$.
Substituting $w=Tz$ in \eqref{eq:logphi-infty} and using $w^{-2k}=(2\pi i C)^{2k}z^{2k}$ yields
\eqref{eq:coeff-logphiT}.
\end{proof}

\begin{proof}[Proof of Theorem~\ref{thm:conj1}]
We first prove the identity for $k\ge1$.
By Lemma~\ref{lem:even-coeff-invariance}, for each $k\ge1$,
\[
[z^{2k}]_{\mathrm{as}}\log\varphi(Tz)=[z^{2k}]_{\mathrm{Tay}}\log\varphi(z).
\]
By Lemma~\ref{lem:tay-at-zero}, the right-hand side equals $-\frac{L_+(2k)}{k}$.
By Lemma~\ref{lem:as-Tz}, the left-hand side equals $-\frac{L_+(-2k)}{k}(2\pi i C)^{2k}$.
Equating these expressions and canceling $-1/k$ gives
\[
L_+(-2k)=\frac{L_+(2k)}{(2\pi i C)^{2k}}
\qquad(k\ge1).
\]
For $k=0$ the identity is tautological. For $k\le -1$, write $k=-m$ with $m\ge1$ and use the already proven
case $m\ge1$:
\[
L_+(2m)=\frac{L_+(-2m)}{(2\pi i C)^{-2m}}
\quad\Longleftrightarrow\quad
L_+(-2k)=\frac{L_+(2k)}{(2\pi i C)^{2k}}.
\]
This completes the proof for all $k\in\Z$.
\end{proof}

The only genuinely nonlocal input in the argument is the coefficient extraction at infinity (Lemma~\ref{lem:as-at-infty}), which in our setting follows from admissibility of $\lambda_n=\tau_n^2$ together with the Quine--Heydari--Song zeta-regularization calculus. All other steps are algebraic consequences of the functional equation \cite[(1.6)]{BOCRS} combined with a parity argument in a small sector at $0$.

\section{Admissibility of higher powers of $(\tau_n)_{n\geq 1}$}\label{sec:parity-dichotomy}

Let $(\tau_n)_{n\ge1}$ be the positive zeros from \cite{BOCRS} and set $a_n:=n+\tfrac12$.
For each integer $m\ge1$ define the heat trace
\[
\theta_m(t):=\sum_{n\ge1}e^{-\tau_n^m t}\qquad (t>0).
\]
In this section we will prove the following theorem.

\begin{theorem}[Parity dichotomy]\label{thm:parity-dichotomy}
Let $m\ge1$ be an integer.
\begin{enumerate}[label=\textup{(\alph*)},leftmargin=2.5em]
\item If $m$ is \emph{odd}, then the heat trace $\theta_m(t)$ contains a \emph{nonzero} term of the form
\[
\kappa_m\,t\log t,
\]
and hence $(\tau_n^m)_{n\ge1}$ is \emph{not} admissible in the sense of Quine--Heydari--Song.
In particular, $(\tau_n)_{n\ge1}$ (i.e.\ $m=1$) is not admissible, and neither is $(\tau_n^m)$ for any odd $m\ge3$.
\item If $m$ is \emph{even}, then $(\tau_n^m)_{n\ge1}$ is admissible in the sense of Quine--Heydari--Song.
\end{enumerate}
\end{theorem}

\subsection{Conditions (1) and (3) for all $m$}

By \cite[Theorem~1.2]{BOCRS} there exists an odd power series $\rho$ with nonnegative coefficients,
radius $2$, and $\rho(2)=\tfrac12$ such that
\begin{equation}\label{eq:tau-rho-parity}
\tau_n=a_n-\rho(a_n^{-1}),\qquad n\ge1.
\end{equation}
Since $a_n^{-1}\le 2/3<2$ and $\rho(z)=\sum_{j\ge1}\alpha_j z^j$ has nonnegative coefficients and
converges on $|z|\le 2$ \cite[Thm.~6.1]{BOCRS}, the partial sums are increasing on $[0,2]$
and hence $\rho$ is increasing there.  Therefore
\[
0\le \rho(a_n^{-1})\le \rho(2)=\tfrac12,
\]
so \eqref{eq:tau-rho-parity} yields $\tau_n\ge a_n-\tfrac12=n$ for all $n\ge1$.

In particular, for every integer $m\ge1$ we have $\tau_n^m\ge n^m$, hence $\theta_m(t)$ converges
absolutely for each $t>0$ and satisfies $\theta_m(t)\ll t^{-1/m}$ as $t\to0^+$.
Thus QHS admissibility conditions \textup{(1)} and \textup{(3)} hold for $(\tau_n^m)_{n\ge1}$ for every $m$:
condition \textup{(1)} is immediate since $\tau_n^m>0$, and condition \textup{(3)} follows because
$t^\beta\theta_m(t)\to0$ as $t\to0^+$ for any $\beta>1/m$ (e.g.\ $\beta=2$ works uniformly for all $m\ge1$).
Hence only the \emph{pure-power} expansion requirement (QHS condition \textup{(2)}) remains to be checked.

\subsection{Odd $m$: an explicit $a_n^{-1}$ term in $\tau_n^m$}

Write $m=2\ell+1$ with $\ell\ge0$.
We show that $\tau_n^m$ has an $a_n^{-1}$ term with a coefficient that can be computed explicitly
from the BOCRS Dirichlet series
\[
L_-(s):=\sum_{n\ge1}\frac{(-1)^n}{\tau_n^s}\qquad (\Re s>1),
\]
which is denoted $L_\tau(s)$ in earlier parts of \cite{BOCRS} and is entire in $s$ by \cite[\S10.2]{BOCRS}.

\begin{lemma}[Coefficient of $a_n^{-1}$ in $\tau_n^{2\ell+1}$]\label{lem:a-1-coeff}
Let $m=2\ell+1$ be odd. Then as $n\to\infty$ one has an asymptotic expansion
\begin{equation}\label{eq:tau-m-expansion}
\tau_n^m
= a_n^m+\sum_{j=1}^{\ell} c_{m,j}\,a_n^{m-2j}\;+\; d_m\,a_n^{-1}\;+\;O(a_n^{-3}),
\end{equation}
for some real constants $c_{m,j}$, where the coefficient $d_m$ is given by the closed formula
\begin{equation}\label{eq:d_m-formula}
d_m
=(-1)^{\ell}\,\frac{L_-(m)}{2^{m-1}\,\pi^{m+1}\,C^{m}}.
\end{equation}
In particular, $d_m\neq0$ for every odd $m\ge1$.
\end{lemma}

\begin{proof}
Define $\tau(z):=z^{-1}-\rho(z)$, so that \eqref{eq:tau-rho-parity} reads $\tau_n=\tau(a_n^{-1})$.
Then $\tau(z)^m$ is a Laurent series at $z=0$ with leading term $z^{-m}$, and
\[
\tau(a_n^{-1})^m = a_n^m+\cdots + \bigl([z^1]\tau(z)^m\bigr)a_n^{-1} + O(a_n^{-3}).
\]
Thus $d_m=[z^1]\tau(z)^m$.

To compute this coefficient, we use the auxiliary series $g$ introduced in \cite[\S6]{BOCRS}:
\[
g(w):=\frac{1}{\pi i}\sum_{k\ge0}\frac{2\,L_-(2k-1)}{2k-1}\,w^{2k-1},
\qquad 0<|w|<\tau_1,
\]
and the defining identity for $\rho$ (equation (6.2) in \cite{BOCRS}):
\begin{equation}\label{eq:BOCRS-6.2}
g\!\left(\frac{1}{2\pi i C\,\tau(z)}\right)=\frac{1}{z}\qquad (|z|\ \text{small}).
\end{equation}
\medskip
\noindent\emph{Justification that $f=1/g$ is holomorphic and locally invertible at $0$.}
The $k=0$ term in the definition of $g$ involves $L_-(-1)$, which is defined by analytic continuation
and satisfies $L_-(-1)=-\frac{1}{4C}$ \cite[\S10.2, eq.~(10.1)]{BOCRS}.  Hence
\[
g(w)=\frac{1}{\pi i}\cdot \frac{2L_-(-1)}{-1}\,w^{-1}+O(w)
=\frac{1}{2\pi i C}\,w^{-1}+O(w)\qquad(w\to0),
\]
so $f(w):=1/g(w)$ is holomorphic near $0$ with $f(0)=0$ and $f'(0)=2\pi i C\neq0$.

Set
\[
w(z):=\frac{1}{2\pi i C\,\tau(z)}.
\]
Then \eqref{eq:BOCRS-6.2} becomes $g(w(z))=z^{-1}$. Equivalently, writing $f(w):=1/g(w)$, we have
\begin{equation}\label{eq:inverse-relation}
f(w(z))=z,
\end{equation}
so $w(z)$ is the local inverse at $0$ of $f$.

Now
\[
\tau(z)^m = \frac{1}{(2\pi i C)^m}\,w(z)^{-m},
\]
so $d_m = [z^1]\tau(z)^m = (2\pi i C)^{-m}\,[z^1]w(z)^{-m}$.
\medskip
\noindent\emph{Computation of $d_m$ via residues.}
Since $f$ is holomorphic near $0$ with $f(0)=0$ and $f'(0)\neq0$, the Lagrange--B\"urmann
(residue) formula gives, for $\psi(w)=w^{-m}$,
\[
[z^1]\psi(w(z))=\Res_{w=0}\frac{\psi'(w)}{f(w)}\,dw.
\]
Applying this with $\psi(w)=w^{-m}$ (so $\psi'(w)=-m w^{-m-1}$) yields
\[
[z^1]\,w(z)^{-m}
=\Res_{w=0}\left(-m w^{-m-1}\right)\frac{dw}{f(w)}
=\Res_{w=0}\left(-m w^{-m-1}\right)g(w)\,dw
= -m\,[w^m]\,g(w).
\]
Hence
\[
d_m
=(2\pi i C)^{-m}[z^1]\,w(z)^{-m}
= -\frac{m}{(2\pi i C)^m}\,[w^m]g(w).
\]
Since $g(w)$ is an odd Laurent series, $[w^m]g(w)=0$ for even $m$, and for odd $m$ we have
\[
[w^m]g(w)=\frac{1}{\pi i}\cdot\frac{2\,L_-(m)}{m}.
\]
Substituting gives
\[
d_m
= -\frac{m}{(2\pi i C)^m}\cdot \frac{1}{\pi i}\cdot \frac{2 L_-(m)}{m}
= -\frac{2L_-(m)}{\pi i}\cdot\frac{1}{(2\pi i C)^m},
\]
which simplifies to \eqref{eq:d_m-formula} because $m=2\ell+1$ implies $i^{m+1}=i^{2\ell+2}=(-1)^{\ell+1}$.

Finally, $L_-(m)\neq0$ for every integer $m\ge1$.
Indeed, $\tau_n$ is strictly increasing because $a_n$ increases and $\rho(a_n^{-1})$
decreases (since $a_n^{-1}$ decreases and $\rho$ is increasing).
Hence $\tau_n^{-m}$ decreases to $0$.
For $m=1$ the series defining $L_-(1)=\sum_{n\ge1}(-1)^n\tau_n^{-1}$ converges by the alternating
series test, and for $m\ge2$ it is absolutely convergent.
In all cases $L_-(m)$ has the sign of its first term, so $L_-(m)<0$ and in particular $L_-(m)\neq0$.

\end{proof}

\subsection{Odd $m$: a harmonic weighted sum produces a logarithm}

Define the harmonic weight sum
\[
H_m(t):=\sum_{n\ge1}a_n^{-1}e^{-a_n^m t}.
\]

\begin{lemma}[Logarithmic asymptotic]\label{lem:Hm-log}
For each integer $m\ge1$,
\[
H_m(t):=\sum_{n\ge1}a_n^{-1}e^{-a_n^m t}
=\frac{1}{m}\log\frac{1}{t}+O(1)\qquad(t\to0^+).
\]
\end{lemma}

\begin{proof}
Fix $c>0$. By Mellin inversion,
\[
e^{-y}=\frac{1}{2\pi i}\int_{\Re s=c}\Gamma(s)\,y^{-s}\,ds\qquad(y>0).
\]
With $y=a_n^m t$ and summing over $n\ge1$ (Fubini is justified since $\sum_{n\ge1}a_n^{-1-ms}$
converges absolutely for $\Re s>0$), we obtain
\[
H_m(t)=\frac{1}{2\pi i}\int_{\Re s=c}\Gamma(s)\,t^{-s}\,S_m(s)\,ds,
\qquad
S_m(s):=\sum_{n\ge1}a_n^{-1-ms}.
\]
Using $\sum_{n\ge0}(n+\tfrac12)^{-p}=\zeta(p,\tfrac12)$ and subtracting the $n=0$ term,
we have the meromorphic identity
\[
S_m(s)=\zeta(1+ms,\tfrac12)-2^{1+ms}.
\]
Since $\zeta(\cdot,\tfrac12)$ has a simple pole at $1$ with residue $1$, it follows that
\[
S_m(s)=\frac{1}{ms}+O(1)\qquad(s\to0),
\]
while $\Gamma(s)=\frac1s+O(1)$ and $t^{-s}=1-s\log t+O(s^2)$. Hence
\[
\Gamma(s)t^{-s}S_m(s)=\frac{1}{m}\frac{1}{s^2}-\frac{1}{m}\frac{\log t}{s}+O(1),
\]
so the coefficient of $1/s$ equals $-\frac{1}{m}\log t+O(1)$.

Now shift the contour from $\Re s=c$ to $\Re s=-\tfrac12$. The only pole crossed is at $s=0$,
and the integral on $\Re s=-\tfrac12$ is $O(t^{1/2})$ by Stirling's bound for $\Gamma$
and polynomial growth of $\zeta$ in vertical strips. Therefore,
\[
H_m(t)= -\frac{1}{m}\log t + O(1)=\frac{1}{m}\log\frac{1}{t}+O(1),
\]
as claimed.
\end{proof}

\subsection{Odd $m$: a nonzero $t\log t$ term in $\theta_m(t)$}

\begin{lemma}[No logarithm for the polynomial part]\label{lem:bn-no-log}
Let $m=2\ell+1$ be odd and define
\[
b_n:=a_n^m+\sum_{j=1}^{\ell} c_{m,j}\,a_n^{m-2j}.
\]
Then the heat trace
\[
\vartheta_m(t):=\sum_{n\ge1}e^{-b_n t}
\]
admits a full asymptotic expansion as $t\to0^+$ in \emph{pure powers} (in particular, it has no
$t\log t$ term).
\end{lemma}

\begin{proof}
Write $u_n:=a_n^{-2}$ and set
\[
F(u):=1+\sum_{j=1}^{\ell}c_{m,j}u^j,
\qquad\text{so that}\qquad
b_n=a_n^m F(u_n)=a_n^m+a_n^{m-2}q(u_n),
\]
where $q(u):=(F(u)-1)/u=\sum_{j=1}^{\ell}c_{m,j}u^{j-1}$ is a polynomial (hence analytic on $|u|<\infty$).
Since $u_n\le 4/9$, $q$ is bounded on $[0,4/9]$.

Now expand
\[
\vartheta_m(t)=\sum_{n\ge1}e^{-a_n^m t}\,e^{-t\,a_n^{m-2}q(u_n)}
\]
by a Taylor expansion in $t$ with a summable remainder (as in the even-power argument):
for each fixed $R\ge0$,
\[
e^{-t\,a_n^{m-2}q(u_n)}=\sum_{r=0}^{R}\frac{(-t)^r}{r!}\,a_n^{(m-2)r}q(u_n)^r+\text{(remainder)}.
\]
Next expand $q(u)^r=\sum_{k=0}^{r(\ell-1)} c_{r,k}u^k$ and note that every resulting weight exponent
\[
(m-2)r-2k \ \ge\ (m-2)r-2r(\ell-1)=r\ \ge\ 0.
\]
Hence $\vartheta_m(t)$ is a finite linear combination (up to an arbitrarily small remainder in
powers of $t$) of sums of the form
\[
\sum_{n\ge1}a_n^{q}\,e^{-a_n^m t}\qquad(q\in\Z_{\ge0}).
\]
For such $q$, Mellin inversion shows that the only additional pole coming from the Hurwitz zeta factor
occurs at $s=(q+1)/m>0$ and therefore cannot collide with the poles of $\Gamma(s)$ at $0,-1,-2,\dots$.
Consequently the contour shift produces a full small-$t$ expansion in pure powers and no logarithms.
\end{proof}

\begin{proposition}[Odd powers have a logarithmic term]\label{prop:odd-log}
Let $m$ be odd. Then as $t\to0^+$,
\[
\theta_m(t)=\cdots + \frac{d_m}{m}\,t\log t + O(t),
\]
where $d_m\neq0$ is the coefficient from Lemma~\ref{lem:a-1-coeff}. In particular, $\theta_m(t)$ cannot have
a full asymptotic expansion in pure powers, so $(\tau_n^m)$ is not admissible in the sense of \cite{QHS}.
\end{proposition}

\begin{proof}
Let $m$ be odd. From Lemma~\ref{lem:a-1-coeff}, there exists $N\ge1$ and a constant $C>0$ such that,
for all $n\ge N$,
\[
\tau_n^m = b_n + d_m a_n^{-1} + \varepsilon_n,
\qquad
b_n:=a_n^m+\sum_{j=1}^{(m-1)/2} c_{m,j}\,a_n^{m-2j},
\qquad
|\varepsilon_n|\le C a_n^{-3}.
\]
The finitely many terms $n<N$ contribute a function analytic at $t=0$ (hence a pure power series),
so they do not affect the presence or the coefficient of a $t\log t$ term. We therefore focus on $n\ge N$.

Fix $t\in(0,1]$. Write
\[
e^{-\tau_n^m t}=e^{-b_n t}\,e^{-d_m t a_n^{-1}}\,e^{-\varepsilon_n t}.
\]
Since $a_n^{-1}\le 2/3$ and $t\le1$, the quantity $|d_m|t a_n^{-1}$ is uniformly bounded in $n$.
Using the standard estimate $|e^{-x}-1+x|\ll x^2 e^{|x|}$, we obtain
\[
e^{-d_m t a_n^{-1}} = 1 - d_m t a_n^{-1} + O(t^2 a_n^{-2}),
\]
and similarly (since $|\varepsilon_n|t\ll a_n^{-3}$)
\[
e^{-\varepsilon_n t}=1+O(t a_n^{-3}).
\]
Multiplying these and absorbing cross-terms gives, for $n\ge N$,
\[
e^{-\tau_n^m t}
= e^{-b_n t}\Bigl(1-d_m t a_n^{-1}\Bigr)+O\!\left(e^{-b_n t}\,(t^2 a_n^{-2}+t a_n^{-3})\right).
\]
Summing over $n\ge N$ and using $\sum_{n\ge1}a_n^{-2}<\infty$ and $\sum_{n\ge1}a_n^{-3}<\infty$, we get
\begin{equation}\label{eq:oddlog-mainreduction}
\theta_m(t)
=\sum_{n\ge1}e^{-b_n t} - d_m t\sum_{n\ge1}a_n^{-1}e^{-b_n t} + O(t).
\end{equation}

\medskip
\noindent\emph{Comparison of $\sum a_n^{-1}e^{-b_n t}$ with $H_m(t)$.}
Since $b_n=a_n^m+O(a_n^{m-2})$ as $n\to\infty$, there exist $K>0$ and $N'\ge N$ such that
$|b_n-a_n^m|\le K a_n^{m-2}$ for all $n\ge N'$. For $t\in(0,1]$ and $n\ge N'$,
\[
|e^{-b_n t}-e^{-a_n^m t}|
=e^{-a_n^m t}\bigl|e^{-(b_n-a_n^m)t}-1\bigr|
\le e^{-a_n^m t}\,|b_n-a_n^m|\,t\,e^{|b_n-a_n^m|t}.
\]
Choose $N''\ge N'$ so that $a_n^2\ge 2K$ for all $n\ge N''$. Then for such $n$ we have
$|b_n-a_n^m|t\le \frac12 a_n^m t$, hence
\[
|e^{-b_n t}-e^{-a_n^m t}|\ll t\,a_n^{m-2}\,e^{-\frac12 a_n^m t}.
\]
Therefore
\[
\sum_{n\ge1}a_n^{-1}\bigl(e^{-b_n t}-e^{-a_n^m t}\bigr)
=O(1)+O\!\left(t\sum_{n\ge1}a_n^{m-3}e^{-\frac12 a_n^m t}\right)=O(1),
\]
where the last bound follows by an integral comparison. If $m=1$ the sum is $\ll t\sum a_n^{-2}=O(t)$,
and if $m\ge3$ (still odd) then
$t\sum a_n^{m-3}e^{-c a_n^m t}\ll t\cdot t^{-(m-2)/m}=t^{2/m}=O(1)$.
Consequently,
\[
\sum_{n\ge1}a_n^{-1}e^{-b_n t}=H_m(t)+O(1)\qquad(t\to0^+).
\]
By Lemma~\ref{lem:Hm-log}, $H_m(t)=\frac{1}{m}\log\frac{1}{t}+O(1)$, hence
\[
\sum_{n\ge1}a_n^{-1}e^{-b_n t}=\frac{1}{m}\log\frac{1}{t}+O(1).
\]

Since $\sum_{n\ge1}e^{-b_n t}=\vartheta_m(t)$ has no $t\log t$ term by Lemma~\ref{lem:bn-no-log},
the coefficient of $t\log t$ in $\theta_m(t)$ is exactly $d_m/m$.
Insert this into \eqref{eq:oddlog-mainreduction} to obtain
\[
\theta_m(t)
=\sum_{n\ge1}e^{-b_n t} - d_m t\Bigl(\frac{1}{m}\log\frac{1}{t}+O(1)\Bigr)+O(t)
=\sum_{n\ge1}e^{-b_n t}+\frac{d_m}{m}\,t\log t+O(t).
\]
Since $d_m\neq0$, the $t\log t$ term is nonzero, so QHS admissibility condition (2) fails.
\end{proof}

\subsection{Even $m$: admissibility}

Let $m=2p$ with $p\ge1$.
Write $\rho(z)=z\sigma(z^2)$ as before, so $\sigma$ is analytic in $|u|<4$ and
\[
\tau_n=a_n-a_n^{-1}\sigma(a_n^{-2})=a_n\bigl(1-u_n\sigma(u_n)\bigr),\qquad u_n:=a_n^{-2}.
\]
Set
\[
F_p(u):=\bigl(1-u\sigma(u)\bigr)^{2p},\qquad u\in[0,r_0],\quad r_0:=4/9.
\]
Then $F_p$ is analytic on $|u|<4$ (hence analytic in a neighborhood of $[0,r_0]$) and satisfies
$F_p(0)=1$.  Thus $F_p(u)-1$ vanishes at $u=0$ and we may write
\[
F_p(u)=1+u\,q_p(u),
\qquad
q_p(u):=\frac{F_p(u)-1}{u},
\]
where $q_p$ is analytic on $|u|<4$. In particular,
\begin{equation}\label{eq:even-normal-form-correct}
\tau_n^{2p}=a_n^{2p}F_p(u_n)=a_n^{2p}+a_n^{2p-2}q_p(u_n).
\end{equation}
Define the heat trace
\[
\theta_{2p}(t):=\sum_{n\ge1}e^{-\tau_n^{2p}t}
=\sum_{n\ge1}e^{-a_n^{2p}t}\,e^{-t a_n^{2p-2}q_p(u_n)}.
\]

\begin{lemma}[Uniform exponential Taylor bound]\label{lem:exp-taylor-bound}
Let $R\ge0$.  For all $x\in\C$ and all $t\in[0,1]$,
\[
e^{-tx}=\sum_{r=0}^{R}\frac{(-t)^r}{r!}\,x^r
+O_R\!\left(t^{R+1}|x|^{R+1}e^{t|x|}\right).
\]
\end{lemma}

\begin{proof}
Use the exponential series tail:
\[
e^{-tx}-\sum_{r=0}^{R}\frac{(-t)^r}{r!}x^r
=\sum_{n\ge R+1}\frac{(-t x)^n}{n!}.
\]
Taking absolute values and bounding the tail by the full exponential series gives
\[
\left|e^{-tx}-\sum_{r=0}^{R}\frac{(-t)^r}{r!}x^r\right|
\le \sum_{n\ge R+1}\frac{(t|x|)^n}{n!}
\le \frac{(t|x|)^{R+1}}{(R+1)!}\,e^{t|x|}.
\]
\end{proof}

\medskip
\noindent\emph{Step 1: Taylor expansion in $t$ with a summable remainder.}
Since $q_p$ is continuous on $[0,r_0]$, set $Q_*:=\sup_{u\in[0,r_0]}|q_p(u)|<\infty$ and define
\[
A_n:=a_n^{2p-2}q_p(u_n),\qquad u_n=a_n^{-2}.
\]
Fix an integer $R\ge0$. For $t\in(0,1]$ and every $n$, Taylor's theorem with remainder gives
\[
e^{-tA_n}
=\sum_{r=0}^{R}\frac{(-t)^r}{r!}\,A_n^{r}
+O\!\left(t^{R+1}|A_n|^{R+1}e^{t|A_n|}\right),
\]
where the implied constant depends only on $R$.
Multiplying by $e^{-a_n^{2p}t}$ and summing over $n$ yields
\begin{equation}\label{eq:theta2p-taylor}
\theta_{2p}(t)
=\sum_{r=0}^{R}\frac{(-t)^r}{r!}\,F_r(t) \;+\; \Err_R(t),
\end{equation}
where
\[
F_r(t):=\sum_{n\ge1}a_n^{(2p-2)r}q_p(u_n)^r\,e^{-a_n^{2p}t},
\]
and
\[
\Err_R(t)\ll_{p,R} t^{R+1}\sum_{n\ge1}a_n^{(2p-2)(R+1)}\,e^{-a_n^{2p}t}\,e^{t|A_n|}.
\]
Since $|A_n|\le Q_* a_n^{2p-2}$, we have
\[
e^{-a_n^{2p}t}e^{t|A_n|}
\le \exp\!\left(-t\bigl(a_n^{2p}-Q_*a_n^{2p-2}\bigr)\right)
=\exp\!\left(-t\,a_n^{2p-2}(a_n^2-Q_*)\right).
\]
Choose $n_0$ so that $a_n^2\ge 2Q_*$ for all $n\ge n_0$. Then for $n\ge n_0$,
$a_n^{2p}-Q_*a_n^{2p-2}\ge \tfrac12 a_n^{2p}$, hence
\[
e^{-a_n^{2p}t}e^{t|A_n|}\le e^{-\frac12 a_n^{2p}t}.
\]
Splitting the sum at $n_0$ (finitely many terms for $n<n_0$) gives
\[
\Err_R(t)
=O_{p,R}(t^{R+1})
+O_{p,R}\!\left(t^{R+1}\sum_{n\ge1}a_n^{(2p-2)(R+1)}e^{-\frac12 a_n^{2p}t}\right).
\]
By the integral comparison
\[
\sum_{n\ge1}a_n^{\alpha}e^{-c a_n^{2p}t}
\ll_{\alpha,p,c} t^{-(\alpha+1)/(2p)}\qquad(t\to0^+,\ \alpha\ge0),
\]
we obtain
\[
\Err_R(t)=O\!\left(t^{R+1-( (2p-2)(R+1)+1)/(2p)}\right)
=O\!\left(t^{(R+1)/p-1/(2p)}\right).
\]
In particular, by choosing $R$ large we can make $\Err_R(t)=O(t^{N})$ for any prescribed $N$.

\medskip
\noindent\emph{Step 2: Expand $q_p(u)^r$ in powers of $u$ and reduce to weighted theta sums.}
Fix $r\ge0$. Since $q_p$ is analytic on a disk $|u|<4$, so is $q_p(u)^r$, hence
\[
q_p(u)^r=\sum_{k\ge0} c_{r,k}u^k
\]
with radius of convergence $4$. Choose $\varrho$ with $r_0<\varrho<4$, by Cauchy estimates,
$|c_{r,k}|\ll_{r} \varrho^{-k}$. Since $u_n\le r_0<\varrho$ and $|c_{r,k}|\ll_r \varrho^{-k}$, we have
\[
\sum_{k\ge0}|c_{r,k}|\,u_n^k \ \ll_r\ \sum_{k\ge0}\left(\frac{r_0}{\varrho}\right)^k<\infty
\]
uniformly in $n$, so Tonelli's theorem justifies interchanging the $k$-- and $n$--sums.
Using $u_n=a_n^{-2}\le r_0$, we may interchange sums (absolute convergence) and obtain
\[
F_r(t)=\sum_{k\ge0}c_{r,k}\sum_{n\ge1}a_n^{(2p-2)r-2k}\,e^{-a_n^{2p}t}
=\sum_{k\ge0}c_{r,k}\,\Theta_{p,\, (2p-2)r-2k}(t),
\]
where for $q\in\mathbb{R}$ we set
\[
\Theta_{p,q}(t):=\sum_{n\ge1}a_n^{q}e^{-a_n^{2p}t}.
\]

\medskip
\noindent\emph{Step 3: Mellin analysis of $\Theta_{p,q}(t)$ and absence of logarithms.}
For each fixed $q\in\mathbb{R}$, Mellin inversion gives
\[
\Theta_{p,q}(t)=\frac{1}{2\pi i}\int_{\Re s=c}\Gamma(s)\,t^{-s}
\Bigl(\zeta(2ps-q,\tfrac12)-2^{2ps-q}\Bigr)\,ds,
\]
valid for any $c>(q+1)/(2p)$.
The only pole of $\zeta(\cdot,\tfrac12)$ is at $1$, so the integrand has at most a simple pole at
\[
2ps-q=1 \quad\Longleftrightarrow\quad s=\frac{q+1}{2p},
\]
and the usual simple poles of $\Gamma$ at $s=0,-1,-2,\dots$.
If $q\in2\mathbb{Z}$ then $(q+1)/(2p)$ can never be a nonpositive integer, hence this pole
never collides with a pole of $\Gamma$. Therefore contour shifting yields a full asymptotic
expansion of $\Theta_{p,q}(t)$ as $t\to0^+$ \emph{in pure powers} (no logarithmic terms).

\medskip
\noindent\emph{Step 4: Assemble the expansion for $\theta_{2p}(t)$.}
For each fixed $R$, the finite sum in \eqref{eq:theta2p-taylor} is a finite linear combination of the
functions $F_r(t)$, each of which is an absolutely convergent series of $\Theta_{p,q}(t)$'s with
$q\in2\mathbb{Z}$. Hence it admits a full Poincar\'e expansion in pure powers as $t\to0^+$,
and the remainder $\Err_R(t)$ can be made arbitrarily small (in the sense of powers of $t$)
by taking $R$ large. This proves that $\theta_{2p}(t)$ has a full small-$t$ expansion in pure powers.
Together with the bounds from the beginning of this section, it follows that $(\tau_n^{2p})_{n\ge1}$
is admissible in the sense of Quine--Heydari--Song.
\qed

\section{Consequences of the parity dichotomy}\label{sec:parity-consequences}

Throughout this section we assume the parity dichotomy from Theorem~\ref{thm:parity-dichotomy}:
for each integer $m\ge1$ the heat trace
\[
\theta_m(t):=\sum_{n\ge1}\e^{-\tau_n^m t}\qquad(t>0)
\]
admits a full small-time expansion in \emph{pure powers} when $m$ is even, while for $m$ odd it contains
a \emph{nonzero} logarithmic term $\kappa_m\,t\log t$.

\subsection{Spectral zeta functions attached to $\tau_n^m$}

For every $m\ge1$ define the (initially convergent) spectral zeta function
\begin{equation}\label{eq:def-Zm}
Z_m(s):=\sum_{n\ge1}(\tau_n^m)^{-s}=\sum_{n\ge1}\tau_n^{-ms}=L_+(ms),
\qquad \Ree s>\frac1m.
\end{equation}
(The equality $Z_m(s)=L_+(ms)$ holds initially by definition and then as meromorphic functions once
continuation is established.)

\begin{lemma}[Exponential decay for large $t$]\label{lem:theta-m-exp-decay}
For every $m\ge1$ there exists $C_m>0$ such that for all $t\ge1$,
\[
\theta_m(t)\le C_m\,\e^{-\tau_1^m t}.
\]
\end{lemma}

\begin{proof}
Write $\theta_m(t)=\e^{-\tau_1^m t}\sum_{n\ge1}\e^{-(\tau_n^m-\tau_1^m)t}$.
For $t\ge1$, $\e^{-(\tau_n^m-\tau_1^m)t}\le \e^{-(\tau_n^m-\tau_1^m)}$.
Since $\tau_n\ge n$ (Lemma~\ref{lem:tau-lower}), we have
$\e^{-(\tau_n^m-\tau_1^m)}\le \e^{\tau_1^m}\e^{-n^m}$, and $\sum_{n\ge1}\e^{-n^m}<\infty$.
Absorb the convergent sum into $C_m$.
\end{proof}

\begin{proposition}[Mellin representation]\label{prop:mellin-Zm}
For $\Ree s>\tfrac1m$,
\begin{equation}\label{eq:mellin-Zm}
Z_m(s)=\frac{1}{\Gamma(s)}\int_0^\infty t^{s-1}\theta_m(t)\,dt.
\end{equation}
\end{proposition}

\begin{proof}
For $\Ree s>0$ and $\Lambda>0$,
$\int_0^\infty t^{s-1}\e^{-\Lambda t}\,dt=\Gamma(s)\Lambda^{-s}$.
For $\Ree s>\tfrac1m$, the series $Z_m(s)=\sum\tau_n^{-ms}$ converges absolutely, hence by Fubini,
\[
\int_0^\infty t^{s-1}\theta_m(t)\,dt
=\sum_{n\ge1}\int_0^\infty t^{s-1}\e^{-\tau_n^m t}\,dt
=\Gamma(s)\sum_{n\ge1}\tau_n^{-ms}=\Gamma(s)Z_m(s).
\]
\end{proof}

\subsection{Even powers: clean QHS consequences}

Fix $p\ge1$ and set $m=2p$. By Theorem~\ref{thm:parity-dichotomy}\textup{(b)},
$(\tau_n^{2p})_{n\ge1}$ is admissible in the sense of \cite{QHS}, hence $\theta_{2p}(t)$ admits a full
small-time expansion in pure powers:
\begin{equation}\label{eq:theta2p-exp}
\theta_{2p}(t)\sim\sum_{\nu=0}^\infty c^{(2p)}_\nu\,t^{j^{(2p)}_\nu},
\qquad j^{(2p)}_0<j^{(2p)}_1<\cdots,\quad j^{(2p)}_\nu\to+\infty,
\end{equation}
with no logarithmic terms.

\begin{theorem}[Meromorphic continuation and residues for $Z_{2p}$]\label{thm:Z2p-meromorphic}
The function $Z_{2p}(s)=L_+(2ps)$ admits a meromorphic continuation to $\C$ with at most simple poles.
More precisely, if $c^{(2p)}_\alpha$ denotes the coefficient of $t^\alpha$ in the expansion
\eqref{eq:theta2p-exp}, then $Z_{2p}$ may have a simple pole at $s=-\alpha$ with residue
\begin{equation}\label{eq:Z2p-res}
\Res_{s=-\alpha} Z_{2p}(s)=\frac{c^{(2p)}_\alpha}{\Gamma(-\alpha)}.
\end{equation}
In particular, although the Mellin analysis produces possible terms $(s+k)^{-1}$ from the coefficient
of $t^k$, the factor $1/\Gamma(s)$ has a simple zero at $s=-k$ and cancels this singularity. Hence, $Z_{2p}$ is holomorphic at every nonpositive integer $s=-k$.

\end{theorem}

\begin{proof}
Fix $N\in\N$ and truncate \eqref{eq:theta2p-exp} at order $t^{j^{(2p)}_{N}}$:
\[
\theta_{2p}(t)=\sum_{\nu=0}^{N-1}c^{(2p)}_\nu\,t^{j^{(2p)}_\nu}+R_N(t),
\qquad R_N(t)=O\!\bigl(t^{j^{(2p)}_{N}}\bigr)\quad(t\to0^+).
\]
Split the Mellin integral \eqref{eq:mellin-Zm} at $1$:
\[
Z_{2p}(s)=\frac{1}{\Gamma(s)}
\Bigl(\int_0^1 t^{s-1}\theta_{2p}(t)\,dt+\int_1^\infty t^{s-1}\theta_{2p}(t)\,dt\Bigr).
\]
By Lemma~\ref{lem:theta-m-exp-decay}, the $[1,\infty)$--integral converges absolutely for every $s\in\C$
and defines an entire function.

On $(0,1]$, insert the truncated expansion and integrate termwise:
\[
\int_0^1 t^{s-1}\theta_{2p}(t)\,dt
=\sum_{\nu=0}^{N-1}\frac{c^{(2p)}_\nu}{s+j^{(2p)}_\nu}+\int_0^1 t^{s-1}R_N(t)\,dt.
\]
The remainder integral is holomorphic for $\Ree s>-j^{(2p)}_{N}$.
Thus $Z_{2p}(s)$ extends meromorphically with possible poles only at $s=-j^{(2p)}_\nu$, all simple,
and \eqref{eq:Z2p-res} follows by inspection since $\Gamma$ is holomorphic and nonzero at $s=-\alpha$
unless $\alpha\in\Z_{\ge0}$.
Finally, for $\alpha=k\in\Z_{\ge0}$ the factor $1/\Gamma(s)$ has a simple zero at $s=-k$,
which cancels the simple pole coming from $(s+k)^{-1}$. Hence $Z_{2p}$ is holomorphic at $s=-k$.
\end{proof}

\begin{theorem}[Special values at nonpositive integers]\label{thm:Z2p-negative}
For every integer $k\ge0$, the meromorphic continuation of $Z_{2p}(s)$ is holomorphic at $s=-k$ and
\begin{equation}\label{eq:Z2p-negative}
Z_{2p}(-k)=(-1)^k k!\,c^{(2p)}_k,
\end{equation}
where $c^{(2p)}_k$ denotes the coefficient of $t^k$ in the small-time expansion \eqref{eq:theta2p-exp}.
Equivalently,
\[
L_+(-2pk)=Z_{2p}(-k)=(-1)^k k!\,c^{(2p)}_k.
\]
\end{theorem}

\begin{proof}
Proceed as in the proof of Theorem~\ref{thm:Z-meromorphic} and Theorem~\ref{thm:Z-negative}:
near $s=-k$ the only potentially singular contribution from the $(0,1]$ integral is $c^{(2p)}_k/(s+k)$.
Since $\Gamma(s)$ has a simple pole at $s=-k$ with residue $\Res_{s=-k}\Gamma(s)=(-1)^k/k!$,
we have
\[
\frac{1}{\Gamma(s)} = (-1)^k k!\,(s+k)+O\!\bigl((s+k)^2\bigr)\qquad(s\to-k),
\]
and therefore $Z_{2p}(s)$ is holomorphic at $s=-k$ and the limit equals \eqref{eq:Z2p-negative}.
\end{proof}

The identity $Z_{2p}(s)=L_+(2ps)$ shows that admissibility of $\tau_n^{2p}$ provides a systematic
\emph{heat-kernel mechanism} for producing the special values $L_+(-2pk)$ (and, via $Z_{2p}'(0)$,
determinant-type constants) from the small-$t$ expansion of $\theta_{2p}(t)$.
For $p=1$, combining this coefficient extraction at infinity with the BOCRS inversion identity
for $\varphi$ yields Conjecture~1 of \cite{BOCRS} (Section~\ref{sec:conj1}).
For $p>1$, admissibility gives the expansion side for $L_+(-2pk)$, but a further symmetry
relating $L_+(-2pk)$ to $L_+(2pk)$ would require additional functional equations for the corresponding
canonical products (currently only available for $p=1$ via $\Phi$ and $\varphi$).

\subsection{Even powers: canonical products and coefficient extraction at infinity}

For $p\ge1$ define the genus-$0$ canonical product
\begin{equation}\label{eq:def-W2p}
W_{2p}(\lambda):=\prod_{n\ge1}\Bigl(1+\frac{\lambda}{\tau_n^{2p}}\Bigr).
\end{equation}
Since $\sum_{n\ge1}\tau_n^{-2p}<\infty$, the product converges absolutely and locally uniformly.
Set also
\[
R_{2p}(\lambda):=\sum_{n\ge1}\frac{1}{\lambda+\tau_n^{2p}}=\frac{W_{2p}'(\lambda)}{W_{2p}(\lambda)}.
\]

\begin{lemma}[Laplace representation]\label{lem:laplace-resolvent-2p}
For $\Ree\lambda>0$,
\[
R_{2p}(\lambda)=\int_0^\infty \e^{-\lambda t}\,\theta_{2p}(t)\,dt.
\]
\end{lemma}

\begin{proof}
Identical to Lemma~\ref{lem:laplace-resolvent}:
$\int_0^\infty\e^{-(\lambda+\tau_n^{2p})t}\,dt=(\lambda+\tau_n^{2p})^{-1}$ for $\Ree\lambda>0$,
and $\sum_{n\ge1}|\lambda+\tau_n^{2p}|^{-1}\ll \sum\tau_n^{-2p}<\infty$ on compact subsets of $\{\Ree\lambda>0\}$,
so Fubini applies.
\end{proof}

\begin{theorem}[Large-$\lambda$ coefficient extraction]\label{thm:W2p-coeff-extraction}
Fix $\delta\in(0,\tfrac{\pi}{2})$ and write $\Sigma_\delta=\{\lambda:\ |\arg\lambda|\le \tfrac{\pi}{2}-\delta\}$.
Then $\log W_{2p}(\lambda)$ admits an asymptotic expansion as $|\lambda|\to\infty$ with $\lambda\in\Sigma_\delta$,
and for every integer $k\ge1$ one has the coefficient identity
\begin{equation}\label{eq:W2p-integer-coeff}
[\lambda^{-k}]\,\log W_{2p}(\lambda)=\frac{(-1)^{k+1}}{k}\,Z_{2p}(-k)
=\frac{(-1)^{k+1}}{k}\,L_+(-2pk).
\end{equation}
\end{theorem}

\begin{proof}
Combine Lemma~\ref{lem:laplace-resolvent-2p} with Watson’s lemma (Lemma~\ref{lem:watson})
applied to $f(t)=\theta_{2p}(t)$, using the small-$t$ expansion \eqref{eq:theta2p-exp} and the
large-$t$ exponential decay from Lemma~\ref{lem:theta-m-exp-decay}. This yields an asymptotic expansion
for $R_{2p}(\lambda)$ in $\Sigma_\delta$. Since $\frac{d}{d\lambda}\log W_{2p}(\lambda)=R_{2p}(\lambda)$,
term-by-term integration along the segment from $1$ to $\lambda$ gives the corresponding expansion for
$\log W_{2p}(\lambda)$, up to an additive constant depending on the branch.

For integer $k\ge1$, the $\lambda^{-k-1}$ coefficient in $R_{2p}(\lambda)$ is
$c^{(2p)}_k\Gamma(k+1)=k!\,c^{(2p)}_k$, hence the $\lambda^{-k}$ coefficient in $\log W_{2p}(\lambda)$ equals
$-\frac{k!\,c^{(2p)}_k}{k}$. Using Theorem~\ref{thm:Z2p-negative}, $Z_{2p}(-k)=(-1)^k k!\,c^{(2p)}_k$,
which yields \eqref{eq:W2p-integer-coeff}.
\end{proof}

\subsection{Odd powers: residues and logarithms}

Fix an odd integer $m=2\ell+1$.
By Theorem~\ref{thm:parity-dichotomy}\textup{(a)}, the heat trace $\theta_m(t)$ has a nonzero term
$\kappa_m\,t\log t$ in its small-$t$ expansion. In the notation of Proposition~\ref{prop:odd-log},
one may take $\kappa_m=\frac{d_m}{m}$, where $d_m\neq0$ is the coefficient of $a_n^{-1}$ in the
large-$n$ expansion of $\tau_n^m$.

\begin{theorem}[A pole at $s=-1$ for $Z_m$]\label{thm:Zm-odd-pole}
Let $m$ be odd and assume that, as $t\to0^+$,
\[
\theta_m(t)=\kappa_m\,t\log t + O(t),
\qquad \kappa_m\neq0.
\]
Then the meromorphic continuation of $Z_m(s)=L_+(ms)$ has a \emph{simple pole} at $s=-1$ with
\begin{equation}\label{eq:Zm-res-minus1}
\Res_{s=-1} Z_m(s)=\kappa_m.
\end{equation}
Equivalently, $L_+(s)$ has a simple pole at $s=-m$ with residue
\begin{equation}\label{eq:Lplus-res-minusm}
\Res_{s=-m} L_+(s)=m\,\kappa_m.
\end{equation}
\end{theorem}

\begin{proof}
Split the Mellin representation \eqref{eq:mellin-Zm} at $1$ as in the proof of
Theorem~\ref{thm:Z-meromorphic}. The $[1,\infty)$ integral is entire by Lemma~\ref{lem:theta-m-exp-decay}.
On $(0,1]$, write $\theta_m(t)=\kappa_m\,t\log t + O(t)$ and compute
\[
\int_0^1 t^{s-1}\,\kappa_m\,t\log t\,dt
=\kappa_m\int_0^1 t^{s}\log t\,dt
=-\frac{\kappa_m}{(s+1)^2}.
\]
Thus the Mellin integral has a double pole at $s=-1$ with principal part $-\kappa_m(s+1)^{-2}$.
Since $\Gamma(s)$ has a simple pole at $s=-1$ with residue $-1$, we have
$1/\Gamma(s)=-(s+1)+O((s+1)^2)$ near $s=-1$.
Multiplying gives
\[
Z_m(s)=\frac{1}{\Gamma(s)}\Bigl(-\frac{\kappa_m}{(s+1)^2}+O\!\Bigl(\frac{1}{s+1}\Bigr)\Bigr)
=\frac{\kappa_m}{s+1}+O(1),
\]
proving \eqref{eq:Zm-res-minus1}. The statement \eqref{eq:Lplus-res-minusm} follows from $Z_m(s)=L_+(ms)$
by the change of variables $u=ms$.
\end{proof}

\begin{corollary}[Recovery of the BOCRS residue formula]\label{cor:BOCRS-residue-from-heat}
For odd $m=2\ell+1$, using $\kappa_m=d_m/m$ from Proposition~\ref{prop:odd-log} yields
\[
\Res_{s=-m}L_+(s)=d_m.
\]
Inserting the explicit value of $d_m$ from Lemma~\ref{lem:a-1-coeff} gives
\[
\Res_{s=-m}L_+(s)=(-1)^{\ell}\,\frac{L_-(m)}{2^{m-1}\,\pi^{m+1}\,C^{m}}
=\frac{2i}{\pi}\,\frac{L_-(m)}{(2\pi i C)^{m}},
\]
which matches the residue identity for $L_+$ recorded in \cite[\S10.2]{BOCRS}.
\end{corollary}

For $m\ge2$, the genus-$0$ product
$W_m(\lambda)=\prod_{n\ge1}(1+\lambda/\tau_n^m)$ is well-defined and satisfies
$\frac{d}{d\lambda}\log W_m(\lambda)=\sum_{n\ge1}(\lambda+\tau_n^m)^{-1}$.
The $t\log t$ term in $\theta_m(t)$ forces new $\log\lambda$ corrections in the large-$\lambda$
asymptotics of the resolvent trace and of $\log W_m(\lambda)$, reflecting a genuinely
\emph{log-polyhomogeneous} Laplace-transform behavior.

\begin{lemma}[Laplace transform of $t\log t$]\label{lem:laplace-tlogt}
For $\Ree\lambda>0$,
\[
\int_0^\infty \e^{-\lambda t}\,t\log t\,dt
=\lambda^{-2}\,\bigl(1-\gamma-\log\lambda\bigr),
\]
where $\gamma$ is Euler's constant and $\log\lambda$ is the principal logarithm.
\end{lemma}

\begin{proof}
For $\Ree\lambda>0$ and $\Ree\alpha>-1$,
$\int_0^\infty \e^{-\lambda t}t^\alpha\,dt=\Gamma(\alpha+1)\lambda^{-\alpha-1}$.
Differentiate in $\alpha$ and evaluate at $\alpha=1$ to obtain
\[
\int_0^\infty \e^{-\lambda t}\,t\log t\,dt
=\Bigl(\Gamma'(2)-\Gamma(2)\log\lambda\Bigr)\lambda^{-2}.
\]
Since $\Gamma(2)=1$ and $\Gamma'(2)=\Gamma(2)\psi(2)=1-\gamma$, the claim follows.
\end{proof}

\begin{proposition}[Logarithms in product asymptotics for odd $m$]\label{prop:odd-loglambda}
Assume $m$ is odd and $\theta_m(t)=\kappa_m\,t\log t+O(t)$ as $t\to0^+$ with $\kappa_m\neq0$.
Define for $\Ree\lambda>0$ the resolvent trace
\[
R_m(\lambda):=\sum_{n\ge1}\frac{1}{\lambda+\tau_n^m}=\int_0^\infty \e^{-\lambda t}\theta_m(t)\,dt.
\]
Then, as $|\lambda|\to\infty$ in any fixed sector $\Sigma_\delta$,
\[
R_m(\lambda)=\cdots -\kappa_m\,\lambda^{-2}\log\lambda + O(|\lambda|^{-2}),
\]
so in particular $R_m(\lambda)$ contains a nonzero $\lambda^{-2}\log\lambda$ term.
Consequently, for the genus-$0$ product $W_m(\lambda)=\prod_{n\ge1}(1+\lambda/\tau_n^m)$ ($m\ge2$),
\[
\log W_m(\lambda)=\cdots + \kappa_m\,\frac{\log\lambda}{\lambda} + O(|\lambda|^{-1})
\qquad (|\lambda|\to\infty,\ \lambda\in\Sigma_\delta),
\]
so $\log W_m(\lambda)$ contains a nonzero $(\log\lambda)/\lambda$ term.
\end{proposition}

\begin{proof}
Split the Laplace integral defining $R_m(\lambda)$ at $1$.
The contribution from $[1,\infty)$ is exponentially small in sectors by Lemma~\ref{lem:theta-m-exp-decay}.
On $(0,1]$, write $\theta_m(t)=\kappa_m\,t\log t + O(t)$ and apply Lemma~\ref{lem:laplace-tlogt}
(and the elementary bound $\int_0^1 \e^{-\lambda t}t\,dt\ll |\lambda|^{-2}$ in sectors).
This yields the stated $-\kappa_m\lambda^{-2}\log\lambda$ term in $R_m(\lambda)$.

Finally, since $R_m(\lambda)=\frac{d}{d\lambda}\log W_m(\lambda)$ away from the poles,
integrate term-by-term along the segment from $1$ to $\lambda$ inside $\Sigma_\delta$,
the integral of $-\kappa_m\lambda^{-2}\log\lambda$ produces $\kappa_m(\log\lambda)/\lambda$ (up to lower-order
$|\lambda|^{-1}$ contributions absorbed in the $O(|\lambda|^{-1})$ term).
\end{proof}

\begin{remark}[Interpretation]
The parity dichotomy can be read as follows.
Even powers $\tau_n^{2p}$ behave, from the heat-kernel/Mellin viewpoint, like a classical
polyhomogeneous spectrum: the associated zeta and determinant calculus is clean, with special values
at nonpositive integers and pure inverse-power asymptotics for $\log W_{2p}$.
Odd powers $\tau_n^{2\ell+1}$ exhibit an intrinsic log-polyhomogeneous obstruction
($t\log t$ in $\theta_m$), and this obstruction is \emph{equivalent} to the existence of
poles of $L_+(s)$ at negative odd integers, with residues governed by the odd Dirichlet series $L_-$
as in \cite[\S10.2]{BOCRS}.
\end{remark}

\bibliographystyle{alpha}
\bibliography{References}

\end{document}